\documentclass[10pt,a4paper]{article}
\usepackage[latin1]{inputenc}
\usepackage{amsmath}
\usepackage{amsfonts}
\usepackage{amssymb}
\usepackage{amsthm}
\usepackage{graphics}
\usepackage{graphicx}
\usepackage[affil-it]{authblk}
\usepackage{tikz}
\usepackage{comment}
\usepackage{fixmath}
\usepackage{subcaption}

\renewcommand{\geq}{\geqslant}
\renewcommand{\leq}{\leqslant}

\newtheorem{theorem}{Theorem}[section]

\newtheorem{lemma}[theorem]{Lemma}

\date{}

\begin{document}
\title{Twofold triple systems with cyclic 2-intersecting Gray codes}

\author{Aras Erzurumluo\u{g}lu\thanks{e-mail address: \texttt{aerzurumluog@mun.ca}}~}

\author{David A.~Pike\thanks{e-mail address: \texttt{dapike@mun.ca}}}

\affil{Department of Mathematics and Statistics \\
       Memorial University of Newfoundland\\
       St.~John's, NL, Canada A1C 5S7}

\date{\today}

\maketitle

\makeatletter

\def\@maketitle{%
  \newpage
  \null
  \vskip 2em%
  \begin{center}%
  \let \footnote \thanks
    {\Large\bfseries \@title \par}%
    \vskip 1.5em%
    {\normalsize
      \lineskip .5em%
      \begin{tabular}[t]{c}%
        \@author
      \end{tabular}\par}%
  \end{center}%
  \par
  \vskip 1.5em}
\makeatother

\begin{abstract}
Given a combinatorial design $\mathcal{D}$ with block set $\mathcal{B}$, the \textit{block-intersection graph} (BIG) of $\mathcal{D}$ is the graph that has $\mathcal{B}$ as its vertex set, where two vertices $B_{1} \in \mathcal{B}$ and $B_{2} \in \mathcal{B} $ are adjacent if and only if $|B_{1} \cap B_{2}| > 0$. The \textit{$i$-block-intersection graph} ($i$-BIG) of $\mathcal{D}$ is the graph that has $\mathcal{B}$ as its vertex set, where two vertices $B_{1} \in \mathcal{B}$ and $B_{2} \in \mathcal{B}$ are adjacent if and only if $|B_{1} \cap B_{2}| = i$. In this paper several constructions are obtained that start with twofold triple systems (TTSs) with Hamiltonian $2$-BIGs and result in larger TTSs that also have Hamiltonian $2$-BIGs. These constructions collectively enable us to determine the complete spectrum of TTSs with Hamiltonian $2$-BIGs (equivalently TTSs with cyclic $2$-intersecting Gray codes)
as well as the complete spectrum for TTSs with $2$-BIGs that have Hamilton paths (i.e., for TTSs with $2$-intersecting Gray codes).

In order to prove these spectrum results, we sometimes require ingredient TTSs that have large partial parallel classes; we prove lower bounds on the sizes of partial parallel clasess in arbitrary TTSs
, and then construct larger TTSs with both cyclic $2$-intersecting Gray codes and parallel classes.


\end{abstract}

\vspace*{\baselineskip}

\vspace*{\baselineskip}

\section{Introduction}

A \textit{combinatorial design} $\mathcal{D}$ is made up of a set $V$ of elements (called \textit{points}), together with a set $\mathcal{B}$ of subsets (called \textit{blocks}) of $V$. A \textit{balanced incomplete block design}, BIBD($v,k,\lambda$), is a combinatorial design in which (i) $|V|=v$, (ii) for each block $B\in \mathcal{B}$, $|B|=k$, and (iii) each 2-subset of $V$ occurs in exactly $\lambda$ blocks of $\mathcal{B}$. In the particularly well-studied case where $k=3$ and $\lambda =1$, the design is called a \textit{Steiner triple system} (STS($v$)); when $k=3$ and $\lambda =2$ it is a \textit{twofold triple system} (TTS($v$)).
More than 70 years ago the existence of TTSs was settled by Bhattacharya~\cite{bhattacharya} who proved that a TTS($v$) exists if and only if $v \equiv 0$ or $1$ (modulo 3), $v \neq 1$. A twofold triple system with block set $\mathcal{B}$ is said to be \textit{simple} if each block in $\mathcal{B}$ is distinct. About 40 years ago van~Buggenhaut ~\cite{buggenhaut} proved that simple twofold triple systems exist if and only if $v \equiv 0$ or $1$ (modulo 3), $v \notin \{1,3\}$ (a different proof of this result can be found in~\cite{stinson}). The more general existence question of simple $\lambda$-fold triple systems was then settled by Dehon~\cite{Dehon1983}.


The \textit{block-intersection graph} (BIG) of a combinatorial design $\mathcal{D}$ with block set $\mathcal{B}$ is the graph that has $\mathcal{B}$ as its vertex set where two vertices $B_{1} \in \mathcal{B}$ and $B_{2} \in \mathcal{B} $ are adjacent if and only if $|B_{1} \cap B_{2}| > 0$. Similarly, the \textit{$i$-block-intersection graph} ($i$-BIG) of $\mathcal{D}$ is the graph that has $\mathcal{B}$ as its vertex set where two vertices $B_{1} \in \mathcal{B}$ and $B_{2} \in \mathcal{B}$ are adjacent if and only if $|B_{1} \cap B_{2}| = i$.
In \cite{dewar} Dewar and Stevens study the structure of codes from a design theoretic perspective. In particular, they define the notion of a $\kappa$-intersecting Gray code (resp.\ cyclic $\kappa$-intersecting Gray code),
which is equivalent to a Hamilton path (resp.\ Hamilton cycle) in the $\kappa$-BIG of the corresponding design.

There is a rich history of studying cycle properties of block-intersection graphs of designs,
dating back to the 1980s when Ron Graham is cited as having asked whether Steiner triple systems
have Hamiltonian BIGs, which is indeed true for any BIBD$(v,k,\lambda)$~\cite{AHM1990,HR1988}.
When $\lambda=1$, the traditional BIG of a design is the same as its 1-BIG, whereas the 1-BIG is a subgraph of the BIG whenever $\lambda \geq 2$.
For $1$-BIGs of triple systems having index $\lambda \geq 2$,
it was proven in~\cite{HPR1999} that every BIBD$(v,3,\lambda)$ with $v \geq 12$ has a Hamiltonian 1-BIG. 
Furthermore, in~\cite{JPS2011} it was shown that for $v \geq 136$, the 1-BIG of every BIBD$(v,4,\lambda)$ is Hamiltonian.
For designs with block sizes $k \in \{5,6\}$ similar results have also been established (as is noted in~\cite{dewar}).

With regard to 2-BIGs and triple systems, 
it is not the case that every TTS$(v)$ has a Hamiltonian 2-BIG.
Mahmoodian observed that the unique (up to isomorphism) TTS on six points has the Petersen graph as its 2-BIG~\cite{mahmoodian},
and even earlier Colbourn and Johnstone had demonstrated a TTS(19) with a connected non-Hamiltonian 2-BIG~\cite{CJ1984}.
It has recently been established that TTSs with connected non-Hamiltonian 2-BIGs exist for most admissible orders;
in particular, the current authors have completely determined the spectrum for 
TTSs that do not have cyclic 2-intersecting Gray codes but which nevertheless have connected 2-BIGs~\cite{firstpaper}.

\begin{theorem}
There exists a \textnormal{TTS($v$)} with a connected non-Hamiltonian \textnormal{$2$-BIG} if and only if $v=6$ or $v \geq 12$, 
$v \equiv 0$ or $1$ (modulo $3$).
\end{theorem}

In this present paper we consider the question of which orders $v$ admit a TTS$(v)$ with a cyclic 2-intersecting Gray code.
Some previous work on this question was done by Dewar and Stevens;
the following result is an immediate consequence of Theorem~5.10 and Theorem~5.11 of~\cite{dewar}.

\begin{theorem} \label{dewar}
There exists a TTS($v$) with a Hamiltonian $2$-BIG if
\begin{itemize}
\item[(i)] $v \equiv 1$ or $4$ (modulo $12$) where $v \not\equiv 0$ (modulo $5$), or
\item[(ii)] $v \equiv 3$ or $7$ (modulo $12$) where $v \geq 7$.
\end{itemize}
\end{theorem}

Since twofold triple systems exist for all $v \equiv 0$ or 1 (modulo 3), many cases are left unresolved by Theorem~\ref{dewar}.
As our main result we completely determine the spectrum of TTSs with Hamiltonian 2-BIGs (i.e., TTSs with cyclic $2$-intersecting Gray codes) by constructing a TTS with a Hamiltonian $2$-BIG for all remaining orders for which such a design exists.
We also determine the spectrum for TTSs 
whose 2-BIGs have a Hamilton path. These results fully solve Problem 5.10 of \cite{dewar}.

In order to achieve these spectrum results, we first usefully manipulate a TTS($v$) with a Hamiltonian 2-BIG and then apply an assortment of techniques to embed the resulting design into a larger design. Then we take advantage of properties of some latin squares and their transversals and of some specific decompositions of certain complete graphs to exhibit a Hamilton cycle in the $2$-BIG of the larger design. At a first glance, the general techniques used in this paper to construct the relevant designs resemble those used in~\cite{firstpaper}.
However the methods that we now employ require additional care and intricacy in order to exhibit
Hamilton cycles in the $2$-BIGs of these delicately constructed designs. The difficulty of constructing TTSs with Hamiltonian $2$-BIGs is hinted at by the fact that the problem of deciding whether a cubic 3-connected graph admits a Hamiltonian cycle is NP-complete (see~\cite{garey}).


We conclude the Introduction with a brief outline of each section of this paper. In Section~\ref{Sec:Definitions} some relevant definitions are stated.
Section~\ref{Sec:$2v+1$ Construction} contains details of a  construction that,
when given a TTS($v$) with a cyclic 2-intersecting Gray code, yields a TTS($2v+1$) that also possesses a cyclic 2-intersecting Gray code.
In Section~\ref{Sec:ppc} we obtain good lower bounds on the size of a maximum partial parallel class in any TTS. These bounds on partial parallel classes are then used in Section~\ref{Sec:$2v+2$ Construction} to find parallel classes in some TTS($2v+1$) with a Hamiltonian $2$-BIG when $v \equiv 1, 4$ (modulo 6). The parallel classes then allow the transition from a $2v+1$ Construction to a $2v+2$ Construction. In Sections \ref{Sec:$3v$ Construction}-\ref{Sec:$3v+3$ Construction} details are given for constructions that take a TTS($v$) with a cyclic 2-intersecting Gray code and yield respectively a TTS($3v$), TTS($3v+1$) and TTS($3v+3$) that also have a cyclic 2-intersecting Gray code. Whereas the $3v$ Construction is valid for all $v$, the $3v+1$ and $3v+3$ Constructions require $v$ to be odd. 
Finally, in Section~\ref{Sec:MainResults} we build some TTSs of small orders that have Hamiltonian $2$-BIGs, and we use these examples together with Theorem \ref{dewar} and the constructions given in the earlier sections to inductively settle the complete spectra of TTSs with $2$-intersecting Gray codes as follows:

\begin{theorem}
Suppose that $v\geq3$. There exists a TTS($v$) with a Hamiltonian $2$-BIG if and only if $v \equiv 0,1$ (modulo $3$), except for $v \in \{3,6\}$.
\end{theorem}

\begin{theorem}
Suppose that $v\geq3$. There exists a TTS($v$) with a $2$-BIG that has a Hamilton path if and only if $v \equiv 0,1$ (modulo $3$), except for $v=3$.
\end{theorem}

\section{Some Definitions}
\label{Sec:Definitions}

In this section we define some notions that will be used throughout the rest of this paper. These definitions can be found in~\cite{thebook} as well.


A \textit{latin square} of \textit{order} $v$ is a $v \times v$ array, each cell of which contains exactly one of the symbols in $\{0,1,\ldots,v-1\}$, with the property that each row and column of the array contains each symbol in $\{0,1,\ldots,v-1\}$ exactly once. A \textit{quasigroup} of \textit{order} $v$ is a pair $(Q, \circ)$, where $Q$ is a set of size $v$ and $\circ$ is a binary operation on $Q$ with the property that for every pair of elements $a,b\in Q$, the equations $a\circ x = b$ and $y \circ a = b$ have unique solutions. We can consider a quasigroup as a latin square with a headline and a sideline. A \textit{transversal} $T$ of a latin square of order $v$ on the symbols $\{0,1,\ldots,v-1\}$ is a set of $v$ cells, exactly one cell from each row and each column, such that each of the symbols in $\{0,1,\ldots,v-1\}$ occurs in exactly one cell of $T$. A \textit{parallel class} in a Steiner triple system with point set $S$ and triple set $\mathcal{T}$
is a subset of $\mathcal{T}$ that partitions $S$. 


\section{\boldmath $2v+1$ Construction}
\label{Sec:$2v+1$ Construction}
	
We provide a construction that generates an infinite family of TTSs with Hamiltonian $2$-BIGs from a given TTS with a Hamiltonian $2$-BIG. We first describe a certain decomposition of the twofold complete graph $2K_{t}$ ($t\geq3$) on the vertex set $\mathbb{Z}_{t-1} \cup \{ \infty \}$ into $t-1$ Hamilton cycles.

\medskip




\textbf{Decomposition}: First suppose that $t$ is even. For each $0\leq j \leq 2t-3$ define a $1$-factor $F_{j}$ of $2K_{t}$ as $F_{j} = \big\{ \{0+j, \infty \}, \{1+j, t-2+j\}, \{2+j, t-3+j \}, \ldots , \{t/2-1+j,t/2+j\} \big\} $ (additions are carried out modulo $t-1$). Note that $F_{k} = F_{\ell}$ if and only if $k \equiv \ell$ (modulo $t-1$), and that $F_{0}, F_{1}, \ldots , F_{2t-3}$ is a $1$-factorization of $2K_{t}$. It is easy to see that for each $0 \leq s \leq t-2$ the edges in the union of $F_{s}$ and $F_{s+1}$ form a Hamilton cycle in $2K_{t}$, call it $H_{s}$. Then $H_{0}, H_{1}, \ldots , H_{t-2}$ is a Hamilton cycle decomposition of $2K_{t}$. 

Now suppose that $t$ is odd. For each $0\leq j \leq (t-3)/2$ define a pair of Hamilton cycles $H_{j}$ and $H_{j}'$ of $2K_{t}$ as follows: $H_{j} = (\infty, j, 1+j, t-2+j, 2+j, t-3+j, \ldots , (t-3)/2+j, (t+1)/2+j, (t-1)/2+j, \infty)$ and $H_{j}' = (\infty, j, t-2+j, 1+j, t-3+j, 2+j, \ldots , (t+1)/2+j, (t-3)/2+j, (t-1)/2+j, \infty)$ (additions are carried out modulo $t-1$). It is easy to check that $H_{0}, H_{1}, \ldots , H_{(t-3)/2}$ and $H_{0}', H_{1}', \ldots , H_{(t-3)/2}'$ are each a Hamilton cycle decomposition of $K_{t}$ (hence $H_{0}, H_{1}, \ldots , H_{(t-3)/2}, H_{0}', H_{1}', \ldots , H_{(t-3)/2}'$ is a Hamilton cycle decomposition of $2K_{t}$). 




\medskip

Now we can state the $2v+1$ Construction.


\noindent\textbf{$\mathbf{2\textit{v}+1}$ Construction}: Let $\mathcal{E}=(V_{1}, \mathcal{B})$ be a TTS($v$) ($v\geq7$ if $v$ is odd, $v\geq 4$ if $v$ is even) on the point set $V_{1}=\{n_{0}, \ldots , n_{v-1}\}$ with a Hamiltonian \textnormal{$2$-BIG} such that $\{n_{3}, n_{4}, n_{5}\}$ is a triple if $v$ is odd, $\{n_{0}, n_{1}, n_{(v/2)+1}\}$ is a triple if $v$ is even. Let $\mathcal{H} = \{H_{0}, H_{1}, \ldots , H_{v-1}\}$ be the Hamilton cycle decomposition of $2K_{v+1}$ as given above (on the vertex set $V_{2}=\mathbb{Z}_{v}\cup \{\infty\}$ where $V_1 \cap V_2 = \emptyset$), for which we define $H_{(v/2)+i}$ to be $H'_{i}$ ($0\leq i \leq (v-2)/2$) when $v$ is even.
Form $\mathcal{E}'=(V', \mathcal{B}')$ where $V'= V_{1} \cup V_{2}$, and define $\mathcal{B}'$ as follows:

\begin{itemize}
\item[$\bullet$] Type 1 triples: Let $\mathcal{B}'$ contain all triples in $\mathcal{B}$, except for $\{n_{3}, n_{4}, n_{5}\}$ if $v$ is odd, $\{n_{0}, n_{1}, n_{(v/2)+1}\}$ if $v$ is even.
\item[$\bullet$] Type 2 triples: For each $0 \leq s \leq v-1$ and for each edge $\{i,j\}$ ($i,j \in V_{2}$) of $H_{s}$, let $\mathcal{B}'$ contain $\{n_{s},i,j\}$ except for the three triples $\{n_{3},2,4\}$, $\{n_{4},2,6\}$, $\{n_{5},4,6\}$ if $v$ is odd, and the three triples $\{n_{0},0,1\}$, $\{n_{1},1,2\}$, $\{n_{(v/2)+1},0,2\}$ if $v$ is even. 
\item[$\bullet$] Type 3 triples: Let $\mathcal{B}'$ contain the three triples $\{2, n_{3}, n_{4}\}$, $\{4, n_{3}, n_{5} \}$ and $\{6, n_{4}, n_{5}\}$ if $v$ is odd, $\{0, n_{0}, n_{(v/2)+1}\}$, $\{1, n_{0}, n_{1} \}$ and $\{2, n_{(v/2)+1}, n_{1}\}$ if $v$ is even.
\item[$\bullet$] Type 4 triples: Let $\mathcal{B}'$ contain the triple $\{2,4,6 \}$ if $v$ is odd, $\{0,1,2\}$ if $v$ is even.
\end{itemize}

Before we state and prove the theorem, we first make the following useful observations about the Type 2 triples.

\noindent\textbf{Remarks}: Consider the $2$-BIG obtained from the Type 2 triples plus the three excluded triples, call it $G$.
$G$ is Hamiltonian, since a Hamilton cycle is given by the following procedure:

Suppose $v$ is odd. Start (step $k=0$) at $\{n_{0}, 0, \infty \}$ 
Set $a_{0}=0$ and $a_{1} = \infty$. Until all triples with $n_{0}$ are visited, at each step $k$ visit $\{n_{0}, a_{k}, a_{k+1}\}$, where $\{a_{k}, a_{k+1}\}$ is the edge in $F_{k}$ (subscript in $F_{k}$ being calculated modulo $2$) that is incident with $a_{k}$. The last triple with $n_{0}$ will be $\{n_{0},2,0\}$. Note that $\{2,0\} \in F_{1} \cap E(H_{0}) \cap E(H_{1})$. Therefore $\{n_{1}, 0,2\} $ is a Type $2$ triple, and moreover in $G$ it is adjacent to $\{n_{0},2,0\}$. Using similar arguments we see that all vertices in $V(G)$ 
are connected through a Hamiltonian path. It is easy to observe that using the above procedure the last triple that will be visited is $\{n_{v-1}, \infty,0\}$ where $\{\infty,0\} \in F_{0} \cap E(H_{v-1}) \cap E(H_{0})$. Moreover $\{n_{0}, \infty, 0\} $ is a Type $2$ triple, and
in $G$ $\{n_{v-1},\infty,0\}$ is adjacent to $\{n_{0}, 0, \infty \}$, hence we have a Hamilton cycle.

We note that excluding the three triples $\{n_{3},2,4\}$, $\{n_{4},2,6\}$, $\{n_{5},4,6\}$ breaks this Hamilton cycle into three paths $P_{1}$, $P_{2}$ and $P_{3}$ where $P_{1}$ has endpoints $\{n_{3},4,\infty\}$ and $\{n_{4},8,2\}$, $P_{2}$ has endpoints $\{n_{4},6,4\}$ and $\{n_{5},8,4\}$, and $P_{3}$ has endpoints $\{n_{5},6,\infty\}$ and $\{n_{3},6,2\}$ (if $v=7$, then $8$ is replaced with $1$).

\medskip

Suppose $v$ is even. Start at $\{n_{v/2}, (v-2)/2, v/2 \}$ 
and follow the triples of the form $\{n_{v/2},x,y\}$ until all triples with $n_{v/2}$ are visited where the pairs of the form $\{x,y\} $ appear as edges in consecutive order in the Hamilton cycle $H'_{0}$, so that the next triple visited is $\{n_{v/2}, (v+2)/2, (v-2)/2\}$ and the last triple including $n_{v/2}$ is $\{n_{v/2}, v/2, \infty\}$. Note that $\{v/2, \infty\} \in E(H'_{0}) \cap E(H_{0})$. So $\{n_{0}, v/2, \infty\}$ is a Type 2 triple. Visit $\{n_{0}, v/2, \infty\}$ next and follow the triples of the form $\{n_{0}, x, y\}$ until all triples with $n_{0}$ are visited, where the pairs $\{x, y\} $ appear in consecutive order in the Hamilton cycle $H_{0}$ so that the next triple visited is $\{n_{0}, \infty, 0\}$ and the last triple including $n_{0}$ is $\{n_{0}, (v+2)/2, v/2\}$. Continue to visit triples by following the edges from the cycles $H_1', H_1, H_2', H_2, \ldots, H_{(v-2)/2}', H_{(v-2)/2}$, so that the edges of $H_i'$ ($i = 1, \ldots, (v-2)/2$) are followed in a consecutive order starting at $\{ v/2, v/2 + 2i-1 \}$ and ending with $\{ v/2, v/2 + 2i \}$ and so that the edges of $H_i$ ($i = 1, \ldots, (v-2)/2$) are followed in a consecutive order starting at $\{ v/2, v/2 +2i \}$ and ending with $\{ v/2, v/2 + 2i+1 \}$
(with arithmetic being done modulo $v$). By combining the edges of $H_i'$ (respectively $H_{i}$) with the point $n_{(v/2)+1}$ (respectively $n_{i}$), this process results in a Hamilton path in $G$. Moreover, the last triple of this Hamilton path is $\{n_{(v/2)-1}, (v-2)/2, v/2 \}$, which is adjacent to $\{n_{v/2}, (v-2)/2, v/2 \}$ in $G$, hence we have a Hamilton cycle.




We note that excluding the three triples $\{n_{0},0,1\}$, $\{n_{(v/2)+1},0,2\}$, $\{n_{1},1,2\}$ breaks this Hamilton cycle into three paths $P_{1}$, $P_{2}$ and $P_{3}$ where $P_{1}$ has endpoints $\{n_{0},1,v-1\}$ and $\{n_{(v/2)+1},1,0\}$, $P_{2}$ has endpoints $\{n_{(v/2)+1},2,v-1\}$ and $\{n_{1},\infty,1\}$, and $P_{3}$ has endpoints $\{n_{1},2,0\}$ and $\{n_{0},\infty,0\}$.


In what follows, the direction in which a path is traversed is tacitly implied by the specified neighbours of the endvertices of the path. 

\begin{theorem} \label{2v+1}
If there exists a \textnormal{TTS($v$)} ($v\geq7$ if $v$ is odd, $v\geq 4$ if $v$ is even) with a Hamiltonian \textnormal{$2$-BIG}, then there exists a \textnormal{TTS($2v+1$)} with a Hamiltonian \textnormal{$2$-BIG}.
\end{theorem}


\begin{proof} First suppose that $v$ is odd. Suppose that $\mathcal{E}=(V_{1}, \mathcal{B})$ is a TTS($v$) on the point set $V_{1}=\{n_{0}, \ldots , n_{v-1}\}$ with a Hamiltonian $2$-BIG such that $\{n_{3}, n_{4}, n_{5}\}$ is a triple. Deleting $\{n_{3}, n_{4}, n_{5}\}$ yields a Hamilton path $P$ in the $2$-BIG where both triples on the ends of the Hamilton path include a different pair chosen from the set $\{n_{3}, n_{4}, n_{5}\}$ and one other point. If necessary, permute the labels $n_{3}, n_{4}, n_{5}$ such that $P$ has ends $\{n_{3}, n_{4}, n_{i} \}$ and $\{n_{4}, n_{5}, n_{j}\}$ for some $0 \leq i,j \leq v-1$, $i,j \notin \{3,4,5\}$. Apply the $2v+1$ Construction. Then a Hamilton cycle in the $2$-BIG of the TTS($2v+1$) is given as follows: 
$( P, \{6,n_{4},n_{5}\}, P_{3},\{2,4,6\}, P_{2}, \{4,n_{3},n_{5}\}, P_{1}, \{2,n_{3},n_{4}\})$.

Now suppose that $v$ is even. Suppose that $\mathcal{E}=(V_{1}, \mathcal{B})$ is a TTS($v$) on the point set $V_{1}=\{n_{0}, \ldots , n_{v-1}\}$ with a Hamiltonian $2$-BIG such that $\{n_{0}, n_{1}, \linebreak n_{(v/2)+1}\}$ is a triple. Deleting $\{n_{0}, n_{1}, n_{(v/2)+1}\}$ yields a Hamilton path $P$ in the $2$-BIG where both triples on the ends of the Hamilton path include a different pair chosen from the set $\{n_{0}, n_{1}, n_{(v/2)+1}\}$ and one other point. If necessary, permute the labels $n_{0}, n_{1}, n_{(v/2)+1}$ such that $P$ has ends $\{n_{0}, n_{(v/2)+1}, n_{i} \}$ and $\{n_{(v/2)+1}, n_{1}, n_{j}\}$ for some $0 \leq i,j \leq v-1$, $i,j \notin \{0,1,(v/2)+1\}$. Apply the $2v+1$ Construction. Then a Hamilton cycle in the $2$-BIG of the TTS($2v+1$) is given as follows: $(P, \{2,n_{(v/2)+1},n_{1}\}, P_{2},\{1,n_{0},n_{1}\}, P_{1}, \{0,1,2\}, P_{3}, \{0,n_{0},n_{(v/2)+1}\})$.
\end{proof}

\section{Large Partial Parallel Classes in TTSs}
\label{Sec:ppc}
We take an approach similar to \cite{lindnerphelps} to obtain a good lower bound on the number of triples in a maximum partial parallel class in a TTS. These three facts are easy to prove (and can be found in \cite{lindnerphelps}).

\begin{itemize}
\item[(1)] If $(S,\mathcal{T})$ is an $S(k,k+1,v)$ Steiner system, then $|\mathcal{T}|=\binom{v}{k}/(k+1)$.
\item[(2)] A partial $S(k,k+1,v)$ Steiner system is a pair $(S,\mathcal{P})$ where $S$ is a $v$-set and $\mathcal{P}$ is a collection of $(k+1)$-element subsets of $S$ such that every $k$-element subset of $S$ belongs to at most one block of $\mathcal{P}$. As a consequence, if $(S,\mathcal{P})$ is a partial $S(k,k+1,v)$ Steiner system, then $|\mathcal{P}|\leq \binom{v}{k}/(k+1)$.
\item[(3)] If $(S,\mathcal{T})$ is a (partial) $S(k,k+1,v)$ Steiner system, then the number of blocks having at least one element in common with a given subset of $k+1$ elements is at most $(k+1)\left[\binom{v-1}{k-1}/k \right]$.
\end{itemize}

Lindner and Phelps \cite{lindnerphelps} used these facts to establish a lower bound on the number of blocks in a maximum partial parallel class in a $S(k,k+1,v)$ Steiner system. Similar lower bounds can be found in \cite{brouwer, gionfriddo, lofaro, mullin, tuza, woolbright}. In the setting of (partial) twofold systems, we have:
\begin{itemize}
\item[($1'$)] If $(S,\mathcal{T})$ is a twofold $S(k,k+1,v)$ system, then $|\mathcal{T}|=2\binom{v}{k}/(k+1)$.
\item[($2'$)] A partial twofold $S(k,k+1,v)$ system is a pair $(S,\mathcal{P})$ where $S$ is a $v$-set and $\mathcal{P}$ is a collection of $(k+1)$-element subsets of $S$ such that every $k$-element subset of $S$ belongs to at most two blocks of $\mathcal{P}$. As a consequence, if $(S,\mathcal{P})$ is a partial twofold $S(k,k+1,v)$ system, then $|\mathcal{P}|\leq 2\binom{v}{k}/(k+1)$.
\item[($3'$)] If $(S,\mathcal{T})$ is a (partial) twofold $S(k,k+1,v)$ system, then the number of blocks having at least one element in common with a given subset of $k+1$ elements is at most $2(k+1)\left[\binom{v-1}{k-1}/k \right]$.

\end{itemize}

Let $\pi$ be a partial parallel class of maximum size in a twofold triple system $(S,\mathcal{T})$ and let $P$ denote the set of points belonging to the triples in $\pi$. Since $\pi$ is a partial parallel class of maximum size every $2$-element subset of $S \setminus P$ 
belongs to two triples of $\mathcal{T}$ each of which intersects $P$ in exactly one point. Denote by $X$ the set of all such intersection points. For each $x \in X$, set $\mathcal{T}(x)= \{B\setminus \{x\}: B\in \mathcal{T}, B \setminus \{x\} \subseteq S\setminus P\}$. Then each $(S \setminus P, \mathcal{T}(x))$ is a partial twofold $S(1,2,v-3|\pi|)$ system, $|\mathcal{T}(x)| \leq 2\binom{v-3|\pi|}{1}/2$ (by ($2'$)), and $\{\mathcal{T}(x)\}_{x\in X}$ is a partition of the $2$-fold multiset of all $2$-element subsets of $S\setminus P$. 

Observe that if $B$ is a triple in $\pi$ containing at least $2$ points of $X$, then for each $x \in X\cap B$ we must have $|\mathcal{T}(x)| \leq 4$. This is because otherwise we can let $y$ be any other point belonging to $X \cap B$ and $D_{1}$ be any $2$-element subset in $\mathcal{T}(y)$, and then (by ($3'$)) at most $4$
of the $2$-element subsets in $\mathcal{T}(x)$ can intersect $D_{1}$, and therefore $\mathcal{T}(x)$ must contain a $2$-element subset $D_{2}$ which is disjoint from $D_{1}$. This results in $(\pi\setminus B) \cup (D_{1}\cup y) \cup (D_{2}\cup x)$ being a partial parallel class in $(S,\mathcal{T})$ larger than $\pi$. Contradiction.

It follows that for each triple $B$ in $\pi$ containing at least $2$ points of $X$, \[\sum_{x\in X\cap B} |\mathcal{T}(x)| \leq 12\binom{v-3|\pi|-1}{0}/1.\] If we let $a$ denote the number of triples in $\pi$ containing at least $2$ points of $X$, then $|\pi|-a$ denotes the number of triples in $\pi$ containing at most $1$ point of $X$. Therefore, 
we get \[2\binom{v-3|\pi|}{2} = \left( \sum_{x\in X} |\mathcal{T}(x)|\right) \] \[\leq a\left[12\binom{v-3|\pi|-1}{0}/1 \right] + (|\pi|-a)\left[2\binom{v-3|\pi|}{1}/2 \right]. \]

There are two cases to consider:

\[(I) \hspace{0.2cm} \left[12\binom{v-3|\pi|-1}{0}/1 \right] \leq \left[2\binom{v-3|\pi|}{1}/2 \right]\]
(which is true if and only if $v \geq 12+3|\pi|$).

Then,

\[2\binom{v-3|\pi|}{2} \leq |\pi|\left[2\binom{v-3|\pi|}{1}/2 \right] .\]

This inequality then yields the following lower bound for $|\pi|$:

\[\dfrac{-\sqrt{v^{2}+6v+9}}{24} + \dfrac{7v-3}{24} \leq |\pi|\]

\[\dfrac{v-1}{4}\leq |\pi|\]

We note that in this case taking into account that $|\pi|$ has to be an integer, this inequality implies that $|\pi| \geq (v+8)/6$ when $v \geq 10$ is equivalent to $4$ modulo $6$; and $|\pi| \geq (v+5)/6$ when $v \geq 7$ is equivalent to $1$ modulo $6$.

\vspace{1cm}

\[(II) \hspace{0.2cm} \left[12\binom{v-3|\pi|-1}{0}/1 \right] > \left[2\binom{v-3|\pi|}{1}/2 \right]\] (which is true if and only if $v < 12+3|\pi|$).

Then,

\[2\binom{v-3|\pi|}{2} < |\pi|\left[12\binom{v-3|\pi|-1}{0}/1 \right].\]

This inequality then yields the following lower bound for $|\pi|$:

\[\dfrac{2v+3}{6} - \dfrac{\sqrt{16v+9}}{6} < |\pi|\]

We note that in this case taking into account that $|\pi|$ has to be an integer, this inequality implies that $|\pi| \geq (v+8)/6$ when $v \geq 16$ is equivalent to $4$ modulo $6$; and $|\pi| \geq (v+5)/6$ when $v \geq 7$ is equivalent to $1$ modulo $6$.

The next theorem summarizes the results in the two cases above.

\begin{theorem}
\label{lowerbound}
Any \textnormal{TTS($v$)} has a partial parallel class with at least $min\{\lceil (v-1)/4\rceil, \lceil (2v+3-\sqrt{16v+9})/6 \rceil \}$ triples. In particular, any \textnormal{TTS($v$)} has a partial parallel class with at least $(v+8)/6$ triples when $v \geq 16$ is equivalent to $4$ modulo $6$; and with at least $(v+5)/6$ triples when $v \geq 7$ is equivalent to $1$ modulo $6$.

\end{theorem}

Finally we note that the work in this section can be generalized without too much difficulty in order to find in a TTS many partial parallel classes of large size.
\iffalse

\bigskip

Consider the three partial parallel classes of $(S,\mathcal{T})$: $\pi_{0} \cup B_{0}, \pi_{1} \cup B_{1}$ and $\pi_{2} \cup B_{2}$ where $B_{0}=\{x,y,a\}, B_{1}=\{x,z,b\}$ and $B_{2}=\{y,z,c\}$. Each of these partial parallel classes have size at least $\lfloor v/6 \rfloor$ whenever $v\geq37$. Now we slightly modify the above arguments to find more partial parallel classes in $(S,\mathcal{T})$ that are of size at least $\lfloor v/6 \rfloor$ whenever $v$ is sufficiently large.

Let $(S,\mathcal{T}_{n-1})$ ($n \geq 4$) be a partial twofold triple system that is obtained by removing the triples in $\pi_{0} \cup B_{0}, \pi_{1} \cup B_{1}, \pi_{2} \cup B_{2}$ and in any $n-4$ additional maximum partial parallel classes $\pi_{3},\pi_{4}, \ldots, \pi_{n-2}$ from the twofold triple system $(S,\mathcal{T}_{0})$; obviously $|\pi_{i}| \leq v/3$ for any partial parallel class. Form $(S,\mathcal{T}_{n})$ ($n \geq 4$) by deleting the triples in a maximum partial parallel class $\pi_{n-1}$ of $(S,\mathcal{T}_{n-1})$. Let $\pi_{n}$ be a partial parallel class of maximum size in $(S,\mathcal{T}_{n})$ and let $P_{n}$ denote the set of points belonging to the triples in $\pi_{n}$. Since $\pi_{n}$ is a partial parallel class of maximum size every $2$-element subset of $S \setminus P_{n}$, except for the at most $3n(v/3)$ $2$-element subsets in $(\pi_{0} \cup B_{0}) \cup (\pi_{1} \cup B_{1}) \cup (\pi_{2} \cup B_{2}) \cup \pi_{3} \cup \ldots \cup \pi_{n-1}$, belongs to two triples of $\mathcal{T}_{n}$ each of which intersects $P_{n}$ in exactly one point. Denote by $X_{n}$ the set of all such intersection points. For each $x \in X_{n}$, set $\mathcal{T}_{n}(x)= \{B\setminus \{x\}: B\in \mathcal{T}_{n}, B \setminus \{x\} \subseteq S\setminus P_{n}\}$. Then each $(S \setminus P_{n}, \mathcal{T}_{n}(x))$ is a partial twofold $S(1,2,v-3|\pi_{n}|)$ system, $|\mathcal{T}_{n}(x)| \leq 2\binom{v-3|\pi_{n}|}{1}/2$ [by ($2'$)], and $\{\mathcal{T}_{n}(x)\}_{x\in X_{n}}$ is a partition of the $2$-fold multiset of all $2$-element subsets of $S\setminus P_{n}$ other than the at most $nv$ $2$-element subsets in $\big( (\pi_{0} \cup B_{0})\cup (\pi_{1} \cup B_{1}) \cup (\pi_{2} \cup B_{2})  \cup \pi_{3} \cup \ldots \cup \pi_{n-1}\big)\cap (S\setminus P_{n})$.

Similar to the claim ($\ast$), if $B$ is a triple in $\pi_{n}$ containing at least $2$ points of $X_{n}$, then for each $x \in X_{n}\cap B$ we must have $|\mathcal{T}_{n}(x)| \leq 4$.

It follows that for each triple $B$ in $\pi_{n}$ containing at least $2$ points of $X_{n}$, \[\sum_{x\in X_{n}\cap B} |\mathcal{T}_{n}(x)| \leq 12\binom{v-3|\pi_{n}|-1}{0}/1.\] If we let $a_{n}$ denote the number of triples in $\pi_{n}$ containing at least $2$ points of $X_{n}$, then $|\pi_{n}|-a_{n}$ denotes the number of triples in $\pi_{n}$ containing at most $1$ point of $X_{n}$. Therefore, we get\[2\binom{v-3|\pi_{n}|}{2} \leq \sum_{x\in X_{n}} |\mathcal{T}_{n}(x)| + nv\] \[\leq a_{n}\left[12\binom{v-3|\pi_{n}|-1}{0}/1 \right] + (|\pi_{n}|-a_{n})\left[2\binom{v-3|\pi_{n}|}{1}/2 \right] + nv.\]

There are two cases to consider:

\[(I) \hspace{0.2cm} \left[12\binom{v-3|\pi_{n}|-1}{0}/1 \right] \leq \left[2\binom{v-3|\pi_{n}|}{1}/2 \right]\] (which is true if and only if $v \geq 3|\pi_{n}|+12$).

Then,

\[2\binom{v-3|\pi_{n}|}{2} \leq |\pi_{n}|\left[2\binom{v-3|\pi_{n}|}{1}/2 \right] + nv.\]

This inequality then yields the following lower bound for $|\pi_{n}|$:

\[\dfrac{-\sqrt{v^{2}+48nv+6v+9}}{24} + \dfrac{7v-3}{24} \leq |\pi_{n}|\]

Suppose that $n\leq20$. Then,

\[\dfrac{-\sqrt{v^{2}+966v+9}}{24} + \dfrac{7v-3}{24} \leq |\pi_{n}|.\]

Taking into account that $|\pi_{n}|$ has to be an integer, some arithmetic shows that for the admissible orders $v \in \{94,99,100\}$ and for $v\geq 103$, this implies that $\lfloor v/6 \rfloor \leq |\pi_{n}|$.

\vspace{1cm}

\[(II) \hspace{0.2cm} \left[12\binom{v-3|\pi_{n}|-1}{0}/1 \right] > \left[2\binom{v-3|\pi_{n}|}{1}/2 \right]\] (which is true if and only if $v < 3|\pi_{n}|+12$).

Then,

\[2\binom{v-3|\pi_{n}|}{2} < |\pi_{n}|\left[12\binom{v-3|\pi_{n}|-1}{0}/1 \right] + nv.\]

This inequality then yields the following lower bound for $|\pi_{n}|$:

\[\dfrac{-\sqrt{(4n+16)v+9}}{6} + \dfrac{2v+3}{6} < |\pi_{n}|\]

Suppose that $n\leq20$. Then,

\[\dfrac{2v+3}{6} - \dfrac{\sqrt{96v+9}}{6} < |\pi_{n}|.\]

Taking into account that $|\pi_{n}|$ has to be an integer, some arithmetic shows that with $v \geq 74$, this implies that $\lfloor v/6 \rfloor \leq |\pi_{n}|$.

\bigskip

We conclude that whenever $v\in \{94,99,100\}$ or $v\geq 103$, in any twofold triple system of order $v$ we can find at least 21 partial parallel classes $\pi_{0}, \pi_{1}, \ldots, \pi_{20}$ each of size at least $\lfloor v/6 \rfloor$, such that $\{x,y,a\} \in \pi_{0}$, $\{x,z,b\} \in \pi_{1}$ and $\{y,z,c\}\in \pi_{2}$ (where $x,y,z,a,b,c$ are all distinct). [[These partial parallel classes consist of disjoint sets of triples.]]

Let $v$ be odd and $\mathcal{E}=(V_{1}, \mathcal{B})$ be a TTS($v$) ($v\geq7$) 
on the point set $V_{1}=\{n_{0}, \ldots , n_{v-1}\}$ with a Hamiltonian \textnormal{$2$-BIG} such that $\{n_{3}, n_{4}, n_{5}\}$ is a triple. 
Let $\mathcal{H} = \{H_{0}, H_{1}, \ldots , H_{v-1}\}$ be the Hamilton cycle decomposition of $2K_{v+1}$ as given in Section \ref{Sec:$2v+1$ Construction} (on the vertex set $V_{2}=\mathbb{Z}_{v}\cup \{\infty\}$ where $V_1 \cap V_2 = \emptyset$), and let $\mathcal{E}'=(V', \mathcal{B}')$ be the TTS($2v+1$) with a Hamiltonian $2$-BIG on the vertex set $V_{1} \cup V_{2}$ that is obtained by applying the $2v+1$ Construction of Section \ref{Sec:$2v+1$ Construction}. 
Suppose that there are $t$ parallel classes $\pi_{0}, \ldots , \pi_{t-1}$ in $\mathcal{E}'$ such that $\pi_{0}$ contains $\{n_{0},3,v-1\}$,  $\pi_{1}$ contains $\{n_{1},1,3\}$ and $\pi_{2}$ contains $\{n_{v-1},1,v-1\}$. (The $2$-BIG of $\mathcal{E}'$ is connected since it is Hamiltonian, so $\mathcal{B}'$ has no repeated triples and therefore the parallel classes necessarily have no common triples.) Also suppose that $(V_{3}, \mathcal{B}_{3})$ is a TTS($t$) on the vertex set $V_{3}= \{z_{0}, \ldots, z_{t-1}\}$ including the triple $\{z_{0}, z_{1}, z_{2}\}$ (the existence of such a triple can be guaranteed by permuting the point labels) where $V_{3} \cap V_{1}= \emptyset$ and $V_{3} \cap V_{2}= \emptyset$, such that its $2$-BIG is Hamiltonian if $t > 3$; and its $2$-BIG consists of two copies of the triple $\{z_{0}, z_{1}, z_{2}\}$ if $t=3$. Form $\mathcal{E}''=(V'', \mathcal{B}'')$ where $V'= V_{1} \cup V_{2} \cup V_{3}$, and $\mathcal{B}''$ is given as follows:

\begin{itemize}
\item[$\bullet$] Type 1 triples: Let $\mathcal{B}''$ contain all triples in $\mathcal{B}$, except for $\{n_{3}, n_{4}, n_{5}\}$. 
\item[$\bullet$] Type 2 triples: For each $0 \leq s \leq v-1$ and for each edge $\{i,j\}$ ($i,j \in V_{2}$) of $H_{s}$, let $\mathcal{B}''$ contain $\{n_{s},i,j\}$ except for the three triples $\{n_{3},2,4\}$, $\{n_{4},2,6\}$, $\{n_{5},4,6\}$. 
\item[$\bullet$] Type 3 triples: Let $\mathcal{B}''$ contain the three triples $\{2, n_{3}, n_{4}\}$, $\{4, n_{3}, n_{5} \}$ and $\{6, n_{4}, n_{5}\}$. 

\item[$\bullet$] Type 4 triples: Let $\mathcal{B}''$ contain the triple $\{2,4,6 \}$. 

\item[$\bullet$] Type 5 triples: If $t \geq 3$, let $\mathcal{B}''$ contain all but one of the triples of $\mathcal{B}_{3}$, with the one excluded triple being $\{z_{0}, z_{1}, z_{2}\}$
(the existence of such a triple can be guaranteed by permuting the point labels). 

\item[$\bullet$] Type 6 triples: For each $0 \leq i \leq t-1$ and each triple $\{a,b,c\}$ ($a,b,c \in V_{1}\cup V_{2}$) in the parallel class $\pi_{i}$ ($0 \leq i \leq t-1$) replace $\{a,b,c\}$ with the three triples $\{z_{i},a,b\}$,  $\{z_{i},a,c\}$ and $\{z_{i},b,c\}$; if $t\geq3$ exclude $\{z_{0},3,v-1\}$, $\{z_{1},1,3\}$ and $\{z_{2},1,v-1\}$.

\item[$\bullet$] Type 7 triples: Let $\mathcal{B}''$ contain the triples $\{1,3,v-1\}$, $\{z_{0}, z_{1},3\}$, $\{z_{0}, z_{2},v-1\}$ and $\{z_{1}, z_{2},1\}$ (note that $\{1,3,v-1\}$ does not appear elsewhere in $\mathcal{B}''$).

\end{itemize}

\begin{lemma}
\label{constructingparallelclasses} [[Construct the required parallel classes. The ideas probably work only for $v\equiv 1$ (modulo $12$).]]
\end{lemma}

\begin{theorem} \label{2v+1+t}
Suppose that $v \equiv 1$ (modulo $12$) and $v\geq \ldots $. Also let $t\geq3$. 
If there exists a \textnormal{TTS($v$)} and a \textnormal{TTS($t$)} with Hamiltonian \textnormal{$2$-BIGs}, then there exists a \textnormal{TTS($2v+1+t$)} with a Hamiltonian \textnormal{$2$-BIG}.
\end{theorem}

\begin{proof} Apply the $2v+1+t$ Construction. As is shown in the $2v+1$ Construction, the $2$-BIG, call it $G$, obtained from the Type 2 triples plus the three excluded triples is Hamiltonian. In $G$ replacing a triple $\{a,b,c\}$ that belongs to a parallel class $\pi_{i}$ with the three triples $\{z_{i},a,b\}$,  $\{z_{i},a,c\}$ and $\{z_{i},b,c\}$ has the effect of replacing a vertex with a path of length $2$. Hence the resulting structure is still Hamiltonian. Excluding the six triples $\{n_{3},2,4\}$, $\{n_{4},2,6\}$, $\{n_{5},4,6\}$, $\{z_{0},3,v-1\}$, $\{z_{1},1,3\}$ and $\{z_{2},1,v-1\}$ breaks this Hamilton cycle into six paths $P_{1}, P_{2}, P_{3}, P_{4}, P_{5}$ and $P_{6}$. The endpoints of these paths are as follows:

$P_{1}$ has endpoints $\{4, \infty, n_{3}\}$ (if $\{4, \infty, n_{3}\}$ is in some $\pi_{j}$ then the endpoint is $\{4, \pi_{j}, n_{3}\}$) and $\{v-1,2,n_{4}\}$ (if $\{v-1,2,n_{4}\}$ is in some $\pi_{j}$ then the endpoint is $\{\pi_{j},2,n_{4}\}$).

$P_{2}$ has endpoints $\{6, 4, n_{4}\}$ (if $\{6, 4, n_{4}\}$ is in some $\pi_{j}$ then the endpoint is $\{6, 4, \pi_{j}\}$) and $\{v-1,4,n_{5}\}$ (if $\{v-1,4,n_{5}\}$ is in some $\pi_{j}$ then the endpoint is $\{\pi_{j},4,n_{5}\}$).

$P_{3}$ has endpoints $\{6, \infty, n_{5}\}$ (if $\{6, \infty, n_{5}\}$ is in some $\pi_{j}$ then the endpoint is $\{6, \pi_{j}, n_{5}\}$) and $\{6,1,n_{v-1}\}$ (if $\{6, 1, n_{v-1}\}$ is in some $\pi_{j}$ then the endpoint is $\{\pi_{j}, 1, n_{v-1}\}$).

$P_{4}$ has endpoints $\{v-1, n_{v-1}, z_{2}\}$ and $\{1,v-1,n_{0}\}$ (if $\{1,v-1,n_{0}\}$ is in some $\pi_{j}$ then the endpoint is $\{1,v-1,\pi_{j}\}$).

$P_{5}$ has endpoints $\{v-1, n_{0}, z_{0}\}$ and $\{3,n_{1},z_{1}\}$.

$P_{6}$ has endpoints $\{3, v-1, n_{1}\}$ (if $\{3, v-1, n_{1}\}$ is in some $\pi_{j}$ then the endpoint is $\{3, v-1, \pi_{j}\}$) and $\{6,2,n_{3}\}$ (if $\{6,2,n_{3}\}$ is in some $\pi_{j}$ then the endpoint is $\{6,2,\pi_{j}\}$).

Deleting $\{n_{3}, n_{4}, n_{5}\}$ yields a Hamilton path $P'$ in the $2$-BIG of $\mathcal{E}$ where both triples on the ends of the Hamilton path include a different pair chosen from the set $\{n_{3}, n_{4}, n_{5}\}$ and one other point. If necessary, permute the labels $n_{3}, n_{4}, n_{5}$ such that $P'$ has ends $\{n_{3}, n_{4}, n_{i} \}$ and $\{n_{4}, n_{5}, n_{j}\}$ for some $0 \leq i,j \leq v-1$, $i,j \notin \{3,4,5\}$. Similarly if $t>3$, deleting $\{z_{0}, z_{1}, z_{2}\}$ yields a Hamilton path $P''$ in the $2$-BIG of the TTS($t$) $(V_{3}, \mathcal{B}_{3})$ where both triples on the ends of the Hamilton path include a different pair chosen from the set $\{z_{0}, z_{1}, z_{2}\}$ and one other point. If necessary, permute the labels $z_{0}, z_{1}, z_{2}$ such that $P''$ has ends $\{z_{0}, z_{1}, z_{i} \}$ and $\{z_{1}, z_{2}, z_{j}\}$ for some $0 \leq i,j \leq t-1$, $i,j \notin \{0,1,2\}$. If $t=3$, then $P''$ is just the triple $\{z_{0}, z_{1}, z_{2}\}$.

Then a Hamilton cycle in $\mathcal{E}''$ is given as follows: $(\{3,v-1,n_{1}\}, \{1,3,v-1\}, P_{4}, \{z_{0}, z_{2},v-1\}, P_{5}, \{z_{0}, z_{1},3\}, P'', \{z_{1}, z_{2},1\}, \{1,n_{v-1},z_{2}\}, \{6,1,n_{v-1}\}, P_{3}, \linebreak \{6,n_{4},n_{5}\}, P', \{2, n_{3},n_{4}\}, P_{1}, \{4,n_{3},n_{5}\}, P_{2}, \{2,4,6\}, P_{6})$.

\end{proof}
\else

\section{\boldmath$2v+2$ Construction when $v \equiv 1, 4$ (modulo 6)}
\label{Sec:$2v+2$ Construction}

In this section, the results from Section \ref{Sec:ppc} will be employed towards obtaining a TTS($2v+1$) with a parallel class when $v\equiv 1$ or $4$ (modulo $6$), thereby extending our $2v+1$ Construction to a $2v+2$ Construction. We will use the notation $[a,b]$ to denote all integers between $a$ and $b$, including $a$ and $b$.

\medskip

\noindent\textbf{$\mathbf{2\textit{v}+2}$ Construction when $v\equiv 1$ or $4$ (modulo $6$)}: Let $v=6k+1$ (for some $k \geq 1$) or $6k+4$ (for some $k \geq 2$), and $\mathcal{E}=(V_{1}, \mathcal{B})$ be a TTS($v$) on the point set $V_{1}=\{n_{0}, \ldots , n_{v-1}\}$ with a Hamiltonian \textnormal{$2$-BIG} such that $\{n_{p}, n_{q}, n_{r}\}$ is a triple of $\mathcal{B}$. 
Let $\mathcal{H}=H_{0}, \ldots, H_{v-1}$ be a Hamilton cycle decomposition of $2K_{v+1}$ on the vertex set $V_{2}= \mathbb{Z}_{v}\cup \{\infty\}$ (where $V_1 \cap V_2 = \emptyset$) where the Hamilton cycles are defined as follows:

\begin{itemize}
\item[(i)] if $v=6k+4$, then $H_{2i}=H_{i}'$ and $H_{2i+1}=H_{i}''$ where for each $0\leq j \leq (v-2)/2$ $H_{j}'' = (\infty, j, 1+j, v-1+j, 2+j, v-2+j, \ldots , (v-2)/2+j, (v+2)/2+j, v/2+j, \infty)$ and $H_{j}' = (\infty, j, v-1+j, 1+j, v-2+j, 2+j, \ldots , (v+2)/2+j, (v-2)/2+j, v/2+j, \infty)$ (additions are carried out modulo $v$),

\item[(ii)] if $v=6k+1$, then for each $0\leq j \leq 2v-1$ define a $1$-factor $F_{j}$ of $2K_{v+1}$ as $F_{j} = \big\{ \{0+j, \infty \}, \{1+j, v-1+j\}, \{2+j, v-2+j \}, \ldots , \{(v-1)/2+j,(v+1)/2+j\} \big\} $ (additions are carried out modulo $v$). So $F_{k} = F_{\ell}$ if and only if $k \equiv \ell$ (modulo $v$), and $\{F_{0}, F_{1}, \ldots , F_{2v-1}\}$ is a $1$-factorization of $2K_{v+1}$. Then for each $0 \leq s \leq v-1$ let $H_{s}$ be the Hamilton cycle in $2K_{v+1}$ formed by the union of $F_{s}$ and $F_{s+1}$. 
\end{itemize}


Form $\mathcal{E}'=(V', \mathcal{B}')$ where $V'= V_{1} \cup V_{2}$, and define $\mathcal{B}'$ by the following procedure:

\begin{itemize}
\item[$\bullet$] Step 1: Let $\mathcal{B}'$ contain all triples in $\mathcal{B}$, except for $\{n_{p}, n_{q}, n_{r}\}$. 
\item[$\bullet$] Step 2: For each $0 \leq s \leq v-1$ and for each edge $\{i,j\}$ ($i,j \in V_{2}$) of $H_{s}$, let $\mathcal{B}'$ contain $\{n_{s},i,j\}$ except for the three triples 
$\{n_{p},k-2,k-1\}$, $\{n_{q},k-2,k\}$, $\{n_{r},k-1,k\}$ if $v=6k+4$; and except for the three triples $\{n_{p},v-1,3\}$, $\{n_{q},1,3\}$, $\{n_{r},1,v-1\}$ if $v=6k+1$. (The existence of these triples will be guaranteed later by associating $n_{p}$ with a Hamilton cycle that contains the edge $\{k-2,k-1\}$, $n_{q}$ with a Hamilton cycle that contains $\{k-2,k\}$, $n_{r}$ with a Hamilton cycle that contains $\{k-1,k\}$ if $v=6k+4$, etc.)
\item[$\bullet$] Step 3: Let $\mathcal{B}'$ contain the four triples $\{k-2, n_{p}, n_{q}\}$, $\{k-1, n_{p}, n_{r} \}$, $\{k, n_{q}, n_{r}\}$ and $\{k-2,k-1,k\}$ 
if $v=6k+4$; and the four triples $\{1, n_{q}, n_{r}\}$, $\{3, n_{p}, n_{q} \}$, $\{v-1, n_{p}, n_{r}\}$ and $\{1,3,v-1\}$ if $v=6k+1$.

\end{itemize}

At this point we note that $\mathcal{E}'$ is a TTS($2v+1$) and that the $2$-BIG restricted to the triples that are formed in Step 2 together with the triples that are excluded in the same step is Hamiltonian, where a Hamilton cycle $H$ can be found by proceeding similarly as in the ``Remarks" of Section \ref{Sec:$2v+1$ Construction}. We also note that

\begin{itemize}
\item[(i)] if $v=6k+4$ excluding the three triples $\{n_{p},k-2,k-1\}$, $\{n_{q},k-2,k\}$, $\{n_{r},k-1,k\}$ breaks $H$ into three paths $P_{1}$, $P_{2}$ and $P_{3}$ where $P_{1}$ has endpoints $\{n_{p},k-1,k-3\}$ and $\{n_{q},k-1,k-2\}$, $P_{2}$ has endpoints $\{n_{q},k,k-3\}$ and $\{n_{r},\infty,k-1\}$, and $P_{3}$ has endpoints $\{n_{r},k,k-2\}$ and $\{n_{p},\infty,k-2\}$,

\item[(ii)] if $v=6k+1$ excluding the three triples $\{n_{p},v-1,3\}$, $\{n_{q},1,3\}$, $\{n_{r},1,v-1\}$ breaks $H$ into three paths $P_{1}$, $P_{2}$ and $P_{3}$ where $P_{1}$ has endpoints $\{n_{p},3,v-3\}$ and $\{n_{q},\infty,1\}$, $P_{2}$ has endpoints $\{n_{q},3,v-1\}$ and $\{n_{r},v-3,1\}$, and $P_{3}$ has endpoints $\{n_{r},v-1,\infty\}$ and $\{n_{p},1,v-1\}$.

\end{itemize}

Suppose that $\mathcal{E}'$ has a parallel class $\pi$. The next step will enable us to form a TTS($2v+2$) $\mathcal{E}''=(V'', \mathcal{B}'')$ with a Hamiltonian $2$-BIG from $\mathcal{E}'$ by blowing up the triples in $\pi$.

\begin{itemize}
\item[$\bullet$] Step 4: Let $V''=V'\cup \{z\}$ ($z \notin V'$), and let $\mathcal{B}''$ contain all triples in $\mathcal{B}'\setminus \pi$. For each triple $\{a,b,c\}$ in $\pi$, let $\mathcal{B}''$ also contain the three triples $\{a,b,z\}, \{a,c,z\}$ and $\{b,c,z\}$.
\end{itemize}

In what follows first we present an explicit procedure to find the required parallel class $\pi$ in the TTS($2v+1$) $\mathcal{E}'$ and then we prove that the $2$-BIG of the TTS($2v+2$) $\mathcal{E}''$ that is obtained following Steps 1-4 is indeed Hamiltonian.

\medskip

In the case when $v=6k+4$ ($k\geq2$), use Theorem \ref{lowerbound} to find a partial parallel class $\pi_{0}$ in $\mathcal{E}$ with $(v+8)/6$ triples. Let $T=\{n_{p}, n_{q}, n_{r}\}$ be one of the triples in $\pi_{0}$ and let $A$ be the set of all points that occur in a triple of $\pi_{0}$. Let $\mathcal{H}= \{H_{0}, \ldots, H_{v-1}\}$ be the Hamilton cycle decomposition of $2K_{v+1}$ on the vertex set $V_{2}= \mathbb{Z}_{v}\cup \{\infty\}$ where $H_{2i}=H_{i}'$ and $H_{2i+1}=H_{i}''$ that is described in (i) of the $2v+2$ Construction. 
We define the operation 
of ``attaching" a point $n_{x}$ in $V_{1}$ to a Hamilton cycle $H_{y}$ as replacing $H_{y}$ with the $3$-cycles in $\{n_{x} \vee \{a,b\}| \{a,b\} \in E(H_{y}) \}$ (where $\vee$ denotes the \textit{join} operation). For each point $n_{x}$ in $A\setminus T$, attach $n_{x}$ to some distinct $H_{y}$ where $0\leq y \leq v-1$ such that $y \notin [2k-4,4k-2]$, $y\neq 5k+4j+4$, $y\neq 5k+4j+5$ ($0\leq j \leq (k-2)/2-1$) if $k$ is even; $y \notin [2k-4,4k-2]$, $y\neq 5k+4j+3$, $y\neq 5k+4j+4$ ($0\leq j \leq (k-3)/2-1$) and $y\neq 2k-10$ (all subscripts are calculated modulo $v$) if $k$ is odd. Also attach $n_{p}$ to $H_{2k-3}$, $n_{q}$ to $H_{2k-2}$ and $n_{r}$ to $H_{2k-1}$. Note that $\{k-2,k-1\} \in E(H_{2k-3})$, $\{k-2,k\} \in E(H_{2k-2})$ and  $\{k-1,k\} \in E(H_{2k-1})$. Replace the four triples $T, \{n_{p},k-2,k-1\}, \{n_{q}, k-2,k\}, \{n_{r},k-1,k\}$ with the four triples $\{k-2,k-1,k\}, \{k-2,n_{p},n_{q}\}, \{k-1,n_{p},n_{r}\}, \{k,n_{q},n_{r}\}$. Consider the edges $\{k+1+2j,v+k-6-j\} \in E(H_{2(k-2)+j})$ ($0\leq j \leq 2k-2$), $\{5k-1, \infty\} \in E(H_{4k-5})$, $\{3k-2, k-3\} \in E(H_{4k-4})$, $\{3k, k-4\} \in E(H_{4k-3})$,  $\{3k+2, k-5\} \in E(H_{4k-2})$. Also consider the edges $\{k+2\ell+2, 4k+2\ell+2\} \in E(H_{5k+4\ell+4})$ and $\{3k+2\ell+4, 2k+2\ell\} \in  E(H_{5k+4\ell+5})$ ($0\leq \ell \leq (k-2)/2-1$) if $k$ is even; $\{k+2\ell+2, 4k+2\ell+1\} \in E(H_{5k+4\ell+3})$, $\{3k+2\ell+4, 2k+2\ell-1\} \in  E(H_{5k+4\ell+4})$ ($0\leq \ell \leq (k-3)/2-1$) and $\{3k-4, 5k-2\} \in E(H_{2k-10})$ if $k$ is odd (all subscripts are calculated modulo $v$). These edges together with the triple $\{k-2,k-1,k\}$ partition $V_{2}$. (Note that none of these edges include $\{k-2,k-1\} \in E(H_{2k-3})$, $\{k-2,k\} \in E(H_{2k-2})$ and  $\{k-1,k\} \in E(H_{2k-1})$.) Attach each vertex in $V_{1} \setminus (A\setminus T)$ with a distinct Hamilton cycle from the collection $\big\{\{ \bigcup_{j=0}^{2k-2} H_{2(k-2)+j}\}, \{H_{4k-5}, H_{4k-4}, H_{4k-3},  H_{4k-2}\}, \{\bigcup_{\ell=0}^{(k-2)/2-1} H_{5k+4\ell+4}\},
\linebreak \{\bigcup_{\ell=0}^{(k-2)/2-1} H_{5k+4\ell+5}\}\big\}$ if $k$ is even; $\big\{ \{\bigcup_{j=0}^{2k-2} H_{2(k-2)+j}\}, \{H_{4k-5}, H_{4k-4}, H_{4k-3}, \linebreak H_{4k-2}\}, \{\bigcup_{\ell=0}^{(k-3)/2-1} H_{5k+4\ell+3}\}, \{\bigcup_{\ell=0}^{(k-3)/2-1} H_{5k+4\ell+4}\}, \{H_{2k-10} \}\big\}$ if $k$ is odd (all subscripts are calculated modulo $v$). Note that these are the Hamilton cycles that contain the edges considered earlier. Then the triples formed from these edges together with $\pi_{0} \setminus T$ and $\{k-2,k-1,k\} $ yield a parallel class $\pi$ in the TTS($2v+1$) $\mathcal{E}'$. 
\medskip

In the case when $v=6k+1$ ($k \geq 1$), use Theorem \ref{lowerbound} to find a partial parallel class $\pi_{0}$ in $\mathcal{E}$ with $(v+5)/6$ triples. Let $T=\{n_{p}, n_{q}, n_{r}\}$ be one of the triples in $\pi_{0}$ and let $A$ be the set of all points that occur in a triple of $\pi_{0}$. Let $\mathcal{H}=\{H_{0}, \ldots, H_{v-1}\}$ be the Hamilton cycle decomposition of $2K_{v+1}$ on the vertex set $V_{2}= \mathbb{Z}_{v}\cup \{\infty\}$ that is described in (ii) of the $2v+2$ Construction. 
We consider the cases $k$ even and $k$ odd separately.

Suppose $k$ is even. For each point $n_{x}$ in $A\setminus T$, attach $n_{x}$ to some distinct $H_{y}$ where $2\leq y \leq (v+3)/2$ such that $y\neq 4$. Also attach $n_{p}$ to $H_{0}$, $n_{q}$ to $H_{1}$ and $n_{r}$ to $H_{v-1}$. Note that $\{v-1,3\} \in E(H_{0})$, $\{1,3\} \in E(H_{1})$ and  $\{1,v-1\} \in E(H_{v-1})$. Replace the four triples $T, \{n_{p},v-1,3\}, \{n_{q},1,3\}, \{n_{r},1,v-1\}$ with the four triples $\{1,3,v-1\}, \{3,n_{p},n_{q}\}, \{v-1,n_{p},n_{r}\}, \{1,n_{q},n_{r}\}$.
For each $i \in [(v+5)/2,v+1]$ consider the edge $\{3+i, v-3+i\} \in H_{i}$ (all subscripts are calculated modulo $v$), and also consider the edge  $\{5, \infty\} \in H_{4}$.
These edges 
partition $V_{2}$. Note that none of these edges include $\{v-1,3\} \in E(H_{0})$, $\{1,3\} \in E(H_{1})$ and  $\{1,v-1\} \in E(H_{v-1})$. Attach each vertex in $V_{1} \setminus (A\setminus T)$ with a distinct Hamilton cycle from the collection $\{H_{i}| i\in \{4 \}\cup [(v+5)/2,v+1]\}$. Then the triples formed from these edges together with $\pi_{0} \setminus T$ 
yield a parallel class in the TTS($2v+1$) $\mathcal{E}'$. 

Suppose $k$ is odd. Now for each point $n_{x}$ in $A\setminus T$, attach $n_{x}$ to some distinct $H_{y}$ where $y \notin L=\{0,1,4,5,8,9, \ldots , (v-5)/2, (v+1)/2, (v+7)/2, (v+9)/2, (v+7)/2+4, (v+9)/2+4, (v+7)/2+8, (v+9)/2+8, \ldots, v-4, v-3, v-1\}$ if $k \equiv 1$ (modulo $4$); and where $y \notin L=\{0,1,4,5,8,9, \ldots , (v-11)/2, (v-9)/2, (v-5)/2, (v+1)/2, (v+7)/2, (v+11)/2, (v+13)/2, (v+11)/2+4, (v+13)/2+4, (v+11)/2+8, (v+13)/2+8, \ldots, v-4, v-3, v-1\}$ if $k \equiv 3$ (modulo $4$). Also attach $n_{p}$ to $H_{0}$, $n_{q}$ to $H_{1}$ and $n_{r}$ to $H_{v-1}$. Note that $\{v-1,3\} \in E(H_{0})$, $\{1,3\} \in E(H_{1})$ and  $\{1,v-1\} \in E(H_{v-1})$. Replace the four triples $T, \{n_{p},v-1,3\}, \{n_{q},1,3\}, \{n_{r},1,v-1\}$ with the four triples $\{1,3,v-1\}, \{3,n_{p},n_{q}\}, \{v-1,n_{p},n_{r}\}, \{1,n_{q},n_{r}\}$. Let $g: L\rightarrow \mathbb{Z}_{(v+1)/2}$ be a one-to-one function such that for any pair $y_{1},y_{2}\in L$, $y_{1}<y_{2}$ implies $g(y_{1})<g(y_{2})$, so that $g$ is a map from $L$ to $\mathbb{Z}_{(v+1)/2}$ that preserves the order of the elements. 
For each $i \in [0, (v-3)/4]$ consider the edge $\{(v-1)/2+2i,(v+1)/2+2i\} \in H_{g^{-1}(i)}$, 
for each $i \in [(v+1)/4, (v-3)/2]$ consider the edge $\{1+2(i-(v+1)/4),2+2(i-(v+1)/4)\} \in H_{g^{-1}(i)}$ and also consider the edge  $\{ 0, \infty \} \in H_{g^{-1}((v-1)/2)}=H_{v-1}$. These edges partition $V_{2}$. Note that none of these edges include $\{v-1,3\} \in E(H_{0})$, $\{1,3\} \in E(H_{1})$ and  $\{1,v-1\} \in E(H_{v-1})$. 
Attach each vertex in $V_{1} \setminus (A\setminus  T)$ with a distinct Hamilton cycle from the collection $\big\{\{\bigcup_{i=0}^{(v-3)/4} H_{g^{-1}(i)}\}, \{\bigcup_{i=(v+1)/4}^{(v-3)/2} H_{g^{-1}(i)}\}, \{H_{v-1}\}\big\}$. Then the triples formed from these edges together with $\pi_{0} \setminus T$ yield a parallel class in the TTS($2v+1$) $\mathcal{E}'$. 


\begin{theorem} \label{2v+2}
If there exists a \textnormal{TTS($v$)} where $v=6k+1$ \textnormal{($k \geq 1$)} or $v=6k+4$ \textnormal{($k\geq 2$)} with a Hamiltonian \textnormal{$2$-BIG}, then there exists a \textnormal{TTS($2v+2$)} with a Hamiltonian \textnormal{$2$-BIG}.
\end{theorem}

\begin{proof}

Suppose that $\mathcal{E}=(V_{1}, \mathcal{B})$ is a TTS($v$) on the point set $V_{1}=\{n_{0}, \ldots , \linebreak n_{v-1}\}$ with a Hamiltonian $2$-BIG such that $\{n_{p}, n_{q}, n_{r}\}$ is a triple. Deleting $\{n_{p}, n_{q}, n_{r}\}$ yields a Hamilton path $P$ in the $2$-BIG where both triples on the ends of the Hamilton path include a different pair chosen from the set $\{n_{p}, n_{q}, n_{r}\}$ and one other point.

Suppose $v=6k+4$, permute the labels $n_{p}, n_{q}, n_{r}$ such that $P$ has ends $\{n_{q}, n_{r}, n_{i} \}$ and $\{n_{p}, n_{q}, n_{j}\}$ for some $0 \leq i,j \leq v-1$; $i,j \notin \{p,q,r\}$. Apply the $2v+2$ Construction. Let $P', P_{1}', P_{2}'$ and $P_{3}'$ denote the paths that are formed from $P, P_{1}, P_{2}$ and $P_{3}$ respectively after blowing up the triples in $\pi$, where $\pi$ is a parallel class of $\mathcal{E}'$ (the TTS($2v+1$) obtained in the initial step of the $2v+2$ Construction). Then a Hamilton cycle $H$ in the $2$-BIG of the TTS($2v+2$) $\mathcal{E}''$ is given as follows: $(P', \{n_{p},n_{q}, k-2\}, P_{3}', \{k-2,k,z\}, \{k-1,k,z\}, \{k-2,k-1,z\}, P_{1}', \{n_{p},n_{r},k-1\}, P_{2}', \{n_{q},n_{r},k\})$.

Now suppose that $v=6k+1$, permute the labels $n_{p}, n_{q}, n_{r}$ such that $P$ has ends $\{n_{p}, n_{r}, n_{i} \}$ and $\{n_{p}, n_{q}, n_{j}\}$ for some $0 \leq i,j \leq v-1$; $i,j \notin \{p,q,r\}$. Apply the $2v+2$ Construction. Then a Hamilton cycle $H$ in the $2$-BIG of the TTS($2v+2$) $\mathcal{E}''$ is given as follows: $(P', \{n_{p},n_{q}, 3\}, P_{1}', \{n_{q},n_{r},1\}, P_{2}', \{1,3,v-1\}, P_{3}', \{n_{p},n_{r},v-1\})$.

Note that if an endpoint $\{a,b,c\}$ of one of the paths $P, P_{1}, P_{2}$ and $P_{3}$ is in the parallel class $\pi$, then after blowing up the triples in $\pi$, $\{a,b,c\}$ is replaced with $\{a,b,z\}$, $\{a,c,z\}$ and $\{b,c,z\}$ with an appropriate ordering of these triples so that in the $2$-BIG the triple at the end of this modified path is adjacent to the neighbour of $\{a,b,c\}$ in $H$.
\end{proof}

\section{\boldmath$3v$ Construction} \label{Sec:$3v$ Construction}

In this section we present a tripling construction for TTSs with Hamiltonian $2$-BIGs.

\textbf{$\mathbf{3\textit{v}}$ Construction}: Let $\mathcal{A}=(V, \mathcal{B})$ be a TTS($v$) on the point set $V=\mathbb{Z}_{v}$ where $ \{0, 1, v-1 \} \in \mathcal{B}$ (if necessary, the points of $V$ can be permuted so that such a block exists). Suppose that the $2$-BIG of $\mathcal{A}$ is Hamiltonian (so necessarily $v\geq 4$). Let $Q_{1}=(V,\circ_{1})$ and $Q_{2}=(V,\circ_{2})$ be two quasigroups of order $v$ where for $i,j \in \mathbb{Z}_{v}$, $i\circ_{1}j=i+j$ (modulo $v$) and $i\circ_{2}j=i+j+1$ (modulo $v$). Form $\mathcal{A}'=(V', \mathcal{B}')$ where $V'=V\times \{1,2,3\}$, and define $\mathcal{B}'$ as follows:

\begin{itemize}
\item[$\bullet$] Type 1 triples: For each triple $\{x,y,z\} \in \mathcal{B}$, let $ \mathcal{B}'$ contain the three triples $\{(x,1),(y,1),(z,1)\}, \{(x,2),(y,2),(z,2)\}, \{(x,3),(y,3),(z,3)\}$, except for the three triples $\{(0,1), (1,1), (v-1,1)\}, \{(0,2), (1,2), (v-1,2)\},$ $ \{(0,3), (1,3), (v-1,3)\}$.
\item[$\bullet$] Type 2 triples: For each $i,j \in \mathbb{Z}_{v}$, let $\mathcal{B}'$ contain the triple $\{(i,1),(j,2),$ $(i\circ_{2}j,3)\}$.
\item[$\bullet$] Type 3 triples: For each $i,j \in \mathbb{Z}_{v}$, let $\mathcal{B}'$ contain the triple $\{(i,1),(j,2),$ $(i\circ_{1}j,3)\}$, except for the six triples listed below.
\begin{center}
$\begin{array}{ccc}

\{(0, 1), (1, 2), (1, 3)\} &\hspace*{5mm} &\{(1, 1), (0, 2), (1, 3)\} \\
\{(0, 1), (v-1, 2), (v-1, 3)\} && \{(v-1, 1), (0, 2), (v-1, 3)\} \\
\{(1, 1), (v-1, 2), (0, 3)\} && \{(v-1, 1), (1, 2), (0, 3)\} \\

\end{array}$
\end{center}

\item[$\bullet$] Type 4 triples: Let $\mathcal{B}'$ contain the nine triples listed below.
\begin{center}
$\begin{array}{ccc}

\{(0, 1), (1, 1), (1, 3)\}&\hspace*{5mm}& \{(0, 1), (v-1, 1), (v-1, 3)\} \\
\{(1, 1), (v-1, 1), (0, 3)\}&& \{(v-1, 1), (0,2), (1,2)\}\\
\{(1, 1), (0, 2), (v-1, 2)\}&& \{(0, 1), (1, 2), (v-1, 2)\}  \\
\{(1, 2), (0, 3), (1, 3)\}&& \{(0, 2), (1, 3),(v-1, 3)\} \\
         \{(v-1, 2), (0, 3), (v-1, 3)\}&& \\

\end{array}$
\end{center}

\end{itemize}

In what follows we show that the $3v$ Construction can be used to construct from a TTS$(v)$ with a Hamiltonian $2$-BIG a TTS$(3v)$ whose $2$-BIG is Hamiltonian.

First note that Type $1$ triples yield three copies of $\mathcal{A}$ with the triple $\{0,1,v-1\}$ being deleted. Clearly for each such copy the $2$-BIG has a Hamilton path. Also note that the set of six pairwise disjoint paths given as below spans the vertex set of the $2$-BIG of $\mathcal{A}'$ restricted to Type $2$ and Type $3$ triples (see Figure~\ref{fig:hex}).

\begin{figure}[h!]
  \begin{center}
    \begin{tikzpicture}

      \draw [black,solid] (-5,3.5) node {$i$};
      \draw [black,solid] (-4.5,4) node {$j$};

      \draw [black] (-3.25,4) node {$0$};
      \draw [black] (-1.75,4) node {$1$};
      \draw [black] (-0.25,4) node {$2$};
      \draw [black] (1.25,4) node {$3$};
      \draw [black] (2.75,4) node {$4$};
      \draw [black] (4.25,4) node {$5$};
\draw [black] (5.75,4) node {$6$};

      \draw [black] (-5,2.75) node {$0$};
      \draw [black] (-5,1.25) node {$1$};
      \draw [black] (-5,-0.25) node {$2$};
      \draw [black] (-5,-1.75) node {$3$};
      \draw [black] (-5,-3.25) node {$4$};
      \draw [black] (-5,-4.75) node {$5$};
\draw [black] (-5,-6.25) node {$6$};

      \draw [red,thin,dashed] (-3.5,2.5) to (-3.25,3.5);
      \draw [orange,thin,dashed] (-0.5,2.5) to (-0.25,3.5);
      \draw [orange,thin,dashed] (1,2.5) to (1.25,3.5);
      \draw [orange,thin,dashed] (2.5,2.5) to (2.75,3.5);
      \draw [orange,thin,dashed] (4,2.5) to (4.25,3.5);

      \draw [red,thin,dashed] (-3.5,2.5) to (-4,2.75);

      \draw [black,fill] (-3.5,2.5) circle [radius=0.1];
      \draw [thin] (-3.5,2.5) to (-3,3);
      \draw [black,solid] (-3,3) circle [radius=0.1];
         \draw [gray,solid] (-3,3) circle [radius=0.15];


      \draw [black,solid] (-1.5,3) circle [radius=0.1];

      \draw [orange,thin] (-1.5,3) to (-0.5,2.5);

      \draw [black,fill] (-0.5,2.5) circle [radius=0.1];
      \draw [thin] (-0.5,2.5) to (0,3);
      \draw [black,solid] (0,3) circle [radius=0.1];

      \draw [orange,thin] (0,3) to (1,2.5);

      \draw [black,fill] (1,2.5) circle [radius=0.1];
      \draw [thin] (1,2.5) to (1.5,3);
      \draw [black,solid] (1.5,3) circle [radius=0.1];

      \draw [orange,thin] (1.5,3) to (2.5,2.5);

      \draw [black,fill] (2.5,2.5) circle [radius=0.1];
      \draw [thin] (2.5,2.5) to (3,3);
      \draw [black,solid] (3,3) circle [radius=0.1];
      \draw [blue,solid] (4.5,3) circle [radius=0.15];

      \draw [orange,thin] (3,3) to (4,2.5);

      \draw [black,fill] (4,2.5) circle [radius=0.1];
      \draw [thin] (4,2.5) to (4.5,3);
      \draw [black,solid] (4.5,3) circle [radius=0.1];

      \draw [red,thin,dashed] (6,3) to (6.5,2.75);


\draw [orange,thin] (5.5,-0.5) to (4.5,0);
\draw [orange,thin] (5.5,-2) to (4.5,-1.5);
\draw [orange,thin] (5.5,-3.5) to (4.5,-3);
\draw [orange,thin] (5.5,-5) to (4.5,-4.5);
\draw [orange,thin] (5.5,-6.5) to (4.5,-6);
\draw [orange,thin] (4,-6.5) to (3,-6);
\draw [orange,thin] (2.5,-6.5) to (1.5,-6);
\draw [orange,thin] (1,-6.5) to (0,-6);
\draw [orange,thin] (-0.5,-6.5) to (-1.5,-6);
\draw [orange,thin] (-0.5,-6.5) to (0,-4.5);
\draw [orange,thin] (1,-6.5) to (1.5,-4.5);
\draw [orange,thin] (2.5,-6.5) to (3,-4.5);
\draw [orange,thin] (4,-6.5) to (4.5,-4.5);
\draw [black,thin] (5.5,-6.5) to (6,-4.5);
\draw [black,thin] (5.5,-5) to (6,-3);
\draw [black,thin] (5.5,-3.5) to (6,-1.5);
\draw [black,thin] (5.5,-2) to (6,0);
\draw [black,thin] (5.5,-0.5) to (6,1.5);

      \draw [orange,thin] (-1.5,3) to (-2,1);
      \draw [orange,thin] (0,3) to (-.5,1);
      \draw [orange,thin] (1.5,3) to (1,1);
      \draw [orange,thin] (3,3) to (2.5,1);
      \draw [blue,thin] (4.5,3) to (4,1);


      \draw [black,solid] (-3,1.5) circle [radius=0.1];

      \draw [orange,thin] (-3,1.5) to (-2,1);

      \draw [black,fill] (-2,1) circle [radius=0.1];
      \draw [thin] (-2,1) to (-1.5,1.5);
      \draw [black,solid] (-1.5,1.5) circle [radius=0.1];

      \draw [orange,thin] (-1.5,1.5) to (-0.5,1);

      \draw [black,fill] (-0.5,1) circle [radius=0.1];
      \draw [thin] (-0.5,1) to (0,1.5);
      \draw [black,solid] (0,1.5) circle [radius=0.1];

      \draw [orange,thin] (0,1.5) to (1,1);

      \draw [black,fill] (1,1) circle [radius=0.1]; 
      \draw [thin] (1,1) to (1.5,1.5);
      \draw [black,solid] (1.5,1.5) circle [radius=0.1]; 

      \draw [black,fill] (0,-7.5) circle [radius=0.1] node [black,right] {\footnotesize{$\{(i,1),(j,2),(i\circ_{1}j,3)\}$}};
      \draw [black,solid] (0,-8) circle [radius=0.1] node [black,right] {\footnotesize{$\{(i,1),(j,2),(i\circ_{2}j,3)\}$}};

      \draw [orange,thin] (1.5,1.5) to (2.5,1);

      \draw [black,fill] (2.5,1) circle [radius=0.1];
      \draw [thin] (2.5,1) to (3,1.5);
      \draw [black,solid] (3,1.5) circle [radius=0.1];

      \draw [blue,thin] (3,1.5) to (4,1);

      \draw [black,fill] (4,1) circle [radius=0.1];
      \draw [thin] (4,1) to (4.5,1.5);
      \draw [black,solid] (4.5,1.5) circle [radius=0.1];
       \draw [green,solid] (4.5,1.5) circle [radius=0.15];


	  \draw [orange,thin] (-3,1.5) to (-3.5,-0.5);
      \draw [orange,thin] (-1.5,1.5) to (-2,-0.5);
      \draw [orange,thin] (0,1.5) to (-.5,-0.5);
      \draw [orange,thin] (1.5,1.5) to (1,-0.5);
      \draw [blue,thin] (3,1.5) to (2.5,-0.5);
      \draw [green,thin] (4.5,1.5) to (4,-0.5);

      \draw [orange,thin,dashed] (-3.5,-0.5) to (-4,-0.25);

      \draw [black,fill] (-3.5,-0.5) circle [radius=0.1];
      \draw [thin] (-3.5,-0.5) to (-3,0);
      \draw [black,solid] (-3,0) circle [radius=0.1];

      \draw [orange,thin] (-3,0) to (-2,-.5);

      \draw [black,fill] (-2,-.5) circle [radius=0.1];
      \draw [thin] (-2,-.5) to (-1.5,0);
      \draw [black,solid] (-1.5,0) circle [radius=0.1];

      \draw [orange,thin] (-1.5,0) to (-0.5,-.5);

      \draw [black,fill] (-0.5,-.5) circle [radius=0.1];
      \draw [thin] (-0.5,-.5) to (0,0);
      \draw [black,solid] (0,0) circle [radius=0.1];

      \draw [orange,thin] (0,0) to (1,-.5);

      \draw [black,fill] (1,-.5) circle [radius=0.1];
      \draw [thin] (1,-.5) to (1.5,0);
      \draw [black,solid] (1.5,0) circle [radius=0.1];

      \draw [blue,thin] (1.5,0) to (2.5,-.5);

      \draw [black,fill] (2.5,-.5) circle [radius=0.1];
      \draw [thin] (2.5,-.5) to (3,0);
      \draw [black,solid] (3,0) circle [radius=0.1];

      \draw [green,thin] (3,0) to (4,-.5);

      \draw [black,fill] (4,-.5) circle [radius=0.1];
      \draw [thin] (4,-.5) to (4.5,0);
      \draw [black,solid] (4.5,0) circle [radius=0.1];

      \draw [orange,thin,dashed] (6,0) to (6.5,-0.25);

      \draw [orange,thin] (-3,0) to (-3.5,-2);
      \draw [orange,thin] (-1.5,0) to (-2,-2);
      \draw [orange,thin] (0,0) to (-0.5,-2);
      \draw [blue,thin] (1.5,0) to (1,-2);
      \draw [green,thin] (3,0) to (2.5,-2);
      \draw [orange,thin] (4.5,0) to (4,-2);

      \draw [orange,thin,dashed] (-3.5,-2) to (-4,-1.75);

      \draw [black,fill] (-3.5,-2) circle [radius=0.1];
      \draw [thin] (-3.5,-2) to (-3,-1.5);
      \draw [black,solid] (-3,-1.5) circle [radius=0.1];

      \draw [orange,thin] (-3,-1.5) to (-2,-2);

      \draw [black,fill] (-2,-2) circle [radius=0.1];
      \draw [thin] (-2,-2) to (-1.5,-1.5);
      \draw [black,solid] (-1.5,-1.5) circle [radius=0.1];

      \draw [orange,thin] (-1.5,-1.5) to (-0.5,-2);

      \draw [black,fill] (-0.5,-2) circle [radius=0.1];
      \draw [thin] (-0.5,-2) to (0,-1.5);
      \draw [black,solid] (0,-1.5) circle [radius=0.1];

      \draw [blue,thin] (0,-1.5) to (1,-2);

      \draw [black,fill] (1,-2) circle [radius=0.1];
      \draw [thin] (1,-2) to (1.5,-1.5);
      \draw [black,solid] (1.5,-1.5) circle [radius=0.1];

      \draw [green,thin] (1.5,-1.5) to (2.5,-2);

      \draw [black,fill] (2.5,-2) circle [radius=0.1];
      \draw [thin] (2.5,-2) to (3,-1.5);
      \draw [black,solid] (3,-1.5) circle [radius=0.1];

      \draw [orange,thin] (3,-1.5) to (4,-2);

      \draw [black,fill] (4,-2) circle [radius=0.1];
      \draw [thin] (4,-2) to (4.5,-1.5);
      \draw [black,solid] (4.5,-1.5) circle [radius=0.1];

      \draw [orange,thin,dashed] (6,-1.5) to (6.5,-1.75);

      \draw [orange,thin] (-3,-1.5) to (-3.5,-3.5);
      \draw [orange,thin] (-1.5,-1.5) to (-2,-3.5);
      \draw [blue,thin] (0,-1.5) to (-.5,-3.5);
      \draw [green,thin] (1.5,-1.5) to (1,-3.5);
      \draw [orange,thin] (3,-1.5) to (2.5,-3.5);
      \draw [orange,thin] (4.5,-1.5) to (4,-3.5);

      \draw [orange,thin,dashed] (-3.5,-3.5) to (-4,-3.25);

      \draw [black,fill] (-3.5,-3.5) circle [radius=0.1];
      \draw [thin] (-3.5,-3.5) to (-3,-3);
      \draw [black,solid] (-3,-3) circle [radius=0.1];
      \draw [blue,solid] (-3,-4.5) circle [radius=0.15];

      \draw [orange,thin] (-3,-3) to (-2,-3.5);

      \draw [black,fill] (-2,-3.5) circle [radius=0.1];
      \draw [thin] (-2,-3.5) to (-1.5,-3);
      \draw [black,solid] (-1.5,-3) circle [radius=0.1];

      \draw [blue,thin] (-1.5,-3) to (-0.5,-3.5);

      \draw [black,fill] (-0.5,-3.5) circle [radius=0.1];
      \draw [thin] (-0.5,-3.5) to (0,-3);
      \draw [black,solid] (0,-3) circle [radius=0.1];

      \draw [green,thin] (0,-3) to (1,-3.5);

      \draw [black,fill] (1,-3.5) circle [radius=0.1];
      \draw [thin] (1,-3.5) to (1.5,-3);
      \draw [black,solid] (1.5,-3) circle [radius=0.1];

      \draw [orange,thin] (1.5,-3) to (2.5,-3.5);

      \draw [black,fill] (2.5,-3.5) circle [radius=0.1];
      \draw [thin] (2.5,-3.5) to (3,-3);
      \draw [black,solid] (3,-3) circle [radius=0.1];

      \draw [orange,thin] (3,-3) to (4,-3.5);

      \draw [black,fill] (4,-3.5) circle [radius=0.1];
      \draw [thin] (4,-3.5) to (4.5,-3);
      \draw [black,solid] (4.5,-3) circle [radius=0.1];

      \draw [orange,thin,dashed] (6,-3) to (6.5,-3.25);

      \draw [orange,thin] (-3,-3) to (-3.5,-5);
      \draw [blue,thin] (-1.5,-3) to (-2,-5);
      \draw [green,thin] (0,-3) to (-.5,-5);
      \draw [orange,thin] (1.5,-3) to (1,-5);
      \draw [orange,thin] (3,-3) to (2.5,-5);
      \draw [orange,thin] (4.5,-3) to (4,-5);

      \draw [orange,thin,dashed] (-3.5,-5) to (-4,-4.75);

      \draw [black,fill] (-3.5,-5) circle [radius=0.1];
      \draw [thin] (-3.5,-5) to (-3,-4.5);
      \draw [black,solid] (-3,-4.5) circle [radius=0.1];

      \draw [blue,thin] (-3,-4.5) to (-2,-5);

      \draw [black,fill] (-2,-5) circle [radius=0.1];
      \draw [thin] (-2,-5) to (-1.5,-4.5);
      \draw [black,solid] (-1.5,-4.5) circle [radius=0.1];
 \draw [green,solid] (-1.5,-4.5) circle [radius=0.15];
  \draw [orange,solid] (-1.5,-6) circle [radius=0.15];
    \draw [orange,solid] (6,-6) circle [radius=0.15];

      \draw [green,thin] (-1.5,-4.5) to (-0.5,-5);

      \draw [black,fill] (-0.5,-5) circle [radius=0.1];
      \draw [thin] (-0.5,-5) to (0,-4.5);
      \draw [black,solid] (0,-4.5) circle [radius=0.1];

      \draw [orange,thin] (0,-4.5) to (1,-5);

      \draw [black,fill] (1,-5) circle [radius=0.1];
      \draw [thin] (1,-5) to (1.5,-4.5);
      \draw [black,solid] (1.5,-4.5) circle [radius=0.1];

      \draw [orange,thin] (1.5,-4.5) to (2.5,-5);

      \draw [black,fill] (2.5,-5) circle [radius=0.1];
      \draw [thin] (2.5,-5) to (3,-4.5);
      \draw [black,solid] (3,-4.5) circle [radius=0.1];

      \draw [orange,thin] (3,-4.5) to (4,-5);

      \draw [black,fill] (4,-5) circle [radius=0.1];
      \draw [thin] (4,-5) to (4.5,-4.5);
      \draw [black,solid] (4.5,-4.5) circle [radius=0.1];

      \draw [orange,thin,dashed] (6,-4.5) to (6.5,-4.75);


      \draw [red,thin,dashed] (-3,-6) to (-3.25,-7);
      \draw [orange,thin,dashed] (0,-6) to (-0.25,-7);
      \draw [orange,thin,dashed] (1.5,-6) to (1.25,-7);
      \draw [orange,thin,dashed] (3,-6) to (2.75,-7);
      \draw [orange,thin,dashed] (4.5,-6) to (4.25,-7);


      \draw [white,fill] (-3,3) circle [radius=0.085];
      \draw [white,fill] (-1.5,3) circle [radius=0.085];
      \draw [white,fill] (0,3) circle [radius=0.085];
      \draw [white,fill] (1.5,3) circle [radius=0.085];
      \draw [white,fill] (3,3) circle [radius=0.085];
      \draw [white,fill] (4.5,3) circle [radius=0.085];

      \draw [white,fill] (-3,1.5) circle [radius=0.085];
      \draw [white,fill] (-1.5,1.5) circle [radius=0.085];
      \draw [white,fill] (0,1.5) circle [radius=0.085];
      \draw [white,fill] (1.5,1.5) circle [radius=0.085];
      \draw [white,fill] (3,1.5) circle [radius=0.085];
      \draw [white,fill] (4.5,1.5) circle [radius=0.085];

      \draw [white,fill] (-3,0) circle [radius=0.085];
      \draw [white,fill] (-1.5,0) circle [radius=0.085];
      \draw [white,fill] (0,0) circle [radius=0.085];
      \draw [white,fill] (1.5,0) circle [radius=0.085];
      \draw [white,fill] (3,0) circle [radius=0.085];
      \draw [white,fill] (4.5,0) circle [radius=0.085];

      \draw [white,fill] (-3,-1.5) circle [radius=0.085];
      \draw [white,fill] (-1.5,-1.5) circle [radius=0.085];
      \draw [white,fill] (0,-1.5) circle [radius=0.085];
      \draw [white,fill] (1.5,-1.5) circle [radius=0.085];
      \draw [white,fill] (3,-1.5) circle [radius=0.085];
      \draw [white,fill] (4.5,-1.5) circle [radius=0.085];

      \draw [white,fill] (-3,-3) circle [radius=0.085];
      \draw [white,fill] (-1.5,-3) circle [radius=0.085];
      \draw [white,fill] (0,-3) circle [radius=0.085];
      \draw [white,fill] (1.5,-3) circle [radius=0.085];
      \draw [white,fill] (3,-3) circle [radius=0.085];
      \draw [white,fill] (4.5,-3) circle [radius=0.085];

      \draw [white,fill] (-3,-4.5) circle [radius=0.085];
      \draw [white,fill] (-1.5,-4.5) circle [radius=0.085];
      \draw [white,fill] (0,-4.5) circle [radius=0.085];
      \draw [white,fill] (1.5,-4.5) circle [radius=0.085];
      \draw [white,fill] (3,-4.5) circle [radius=0.085];
      \draw [white,fill] (4.5,-4.5) circle [radius=0.085];

      \draw [black,solid] (6,3) circle [radius=0.1];
      \draw [red,solid] (6,3) circle [radius=0.15];

      \draw [white,fill] (6,3) circle [radius=0.085];

      \draw [black,solid] (6,1.5) circle [radius=0.1];
      \draw [violet,solid] (6,1.5) circle [radius=0.15];

      \draw [white,fill] (6,1.5) circle [radius=0.085];

\draw [black,fill] (5.5,-0.5) circle [radius=0.1];
      \draw [orange,thin] (5.5,-0.5) to (6,0);
      \draw [black,solid] (6,0) circle [radius=0.1];
      \draw [white,fill] (6,0) circle [radius=0.085];

\draw [black,fill] (5.5,-2) circle [radius=0.1];
      \draw [orange,thin] (5.5,-2) to (6,-1.5);
      \draw [black,solid] (6,-1.5) circle [radius=0.1];
      \draw [white,fill] (6,-1.5) circle [radius=0.085];

\draw [black,fill] (5.5,-3.5) circle [radius=0.1];
      \draw [orange,thin] (5.5,-3.5) to (6,-3);
      \draw [black,solid] (6,-3) circle [radius=0.1];
      \draw [white,fill] (6,-3) circle [radius=0.085];

\draw [black,fill] (5.5,-5) circle [radius=0.1];
      \draw [orange,thin] (5.5,-5) to (6,-4.5);
      \draw [black,solid] (6,-4.5) circle [radius=0.1];
      \draw [white,fill] (6,-4.5) circle [radius=0.085];

\draw [black,fill] (5.5,-6.5) circle [radius=0.1];
      \draw [orange,thin] (5.5,-6.5) to (6,-6);
      \draw [black,solid] (6,-6) circle [radius=0.1];
      \draw [white,fill] (6,-6) circle [radius=0.085];

\draw [black,fill] (4,-6.5) circle [radius=0.1];
      \draw [thin] (4,-6.5) to (4.5,-6);
      \draw [black,solid] (4.5,-6) circle [radius=0.1];
      \draw [white,fill] (4.5,-6) circle [radius=0.085];

\draw [black,fill] (2.5,-6.5) circle [radius=0.1];
      \draw [thin] (2.5,-6.5) to (3,-6);
      \draw [black,solid] (3,-6) circle [radius=0.1];
      \draw [white,fill] (3,-6) circle [radius=0.085];

\draw [black,fill] (1,-6.5) circle [radius=0.1];
      \draw [thin] (1,-6.5) to (1.5,-6);
      \draw [black,solid] (1.5,-6) circle [radius=0.1];
      \draw [white,fill] (1.5,-6) circle [radius=0.085];

\draw [black,fill] (-0.5,-6.5) circle [radius=0.1];
      \draw [thin] (-0.5,-6.5) to (0,-6);
      \draw [black,solid] (0,-6) circle [radius=0.1];
      \draw [white,fill] (0,-6) circle [radius=0.085];

      \draw [black,solid] (-1.5,-6) circle [radius=0.1];
      \draw [white,fill] (-1.5,-6) circle [radius=0.085];

      \draw [black,solid] (-3,-6) circle [radius=0.1];
      \draw [red,solid] (-3,-6) circle [radius=0.15];

      \draw [white,fill] (-3,-6) circle [radius=0.085];

\draw [black,fill] (5.5,-0.5) circle [radius=0.1];
\draw [black,fill] (5.5,-2) circle [radius=0.1];
\draw [black,fill] (5.5,-3.5) circle [radius=0.1];
\draw [black,fill] (5.5,-5) circle [radius=0.1];
\draw [black,fill] (5.5,-6.5) circle [radius=0.1];

       \end{tikzpicture}
    \caption{Paths \textcolor{gray}{$P_{1}$},
    \textcolor{orange}{$P_{2}$},
    \textcolor{blue}{$P_{3}$},
    \textcolor{green}{$P_{4}$},
    \textcolor{red}{$P_{5}$}, and
    \textcolor{violet}{$P_{6}$} in the substructure induced by Type $2$ and Type $3$ triples (for $v=7$)} \label{fig:hex} 
  \end{center}
\end{figure}
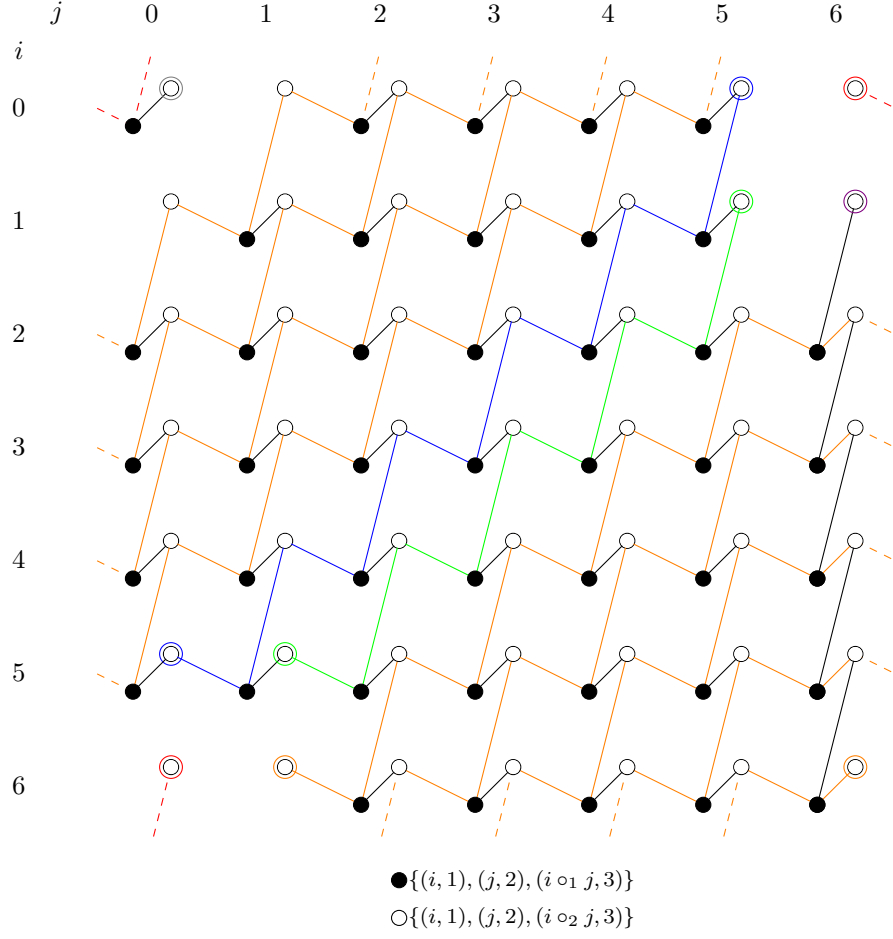

\begin{itemize}

\item[] $P_{1} = \{(0,1), (0,2), (1,3)\}$

\item[] $P_{2} = \{(v-1,1), (1,2),(1,3)\},
\{(v-1,1), (2,2), (1,3)\}, \\
\{(v-2,1), (2,2), (1,3)\},
\ldots,
\{(2,1), (v-2,2), (1,3)\}, \\
\{(2,1), (v-1,2), (1,3)\},
\{(2,1), (v-1,2), (2,3)\},
\{(2,1), (0,2), (2,3)\}, \\
\{(1,1), (0,2), (2,3)\},
\ldots,
\{(3,1), (v-2,2), (2,3)\}, \\
\{(3,1), (v-1,2), (2,3)\},
\{(3,1), (v-1,2), (3,3)\},
\{(3,1), (0,2), (3,3)\}, \\
\{(2,1), (0,2), (3,3)\},
\ldots,
\{(4,1), (v-2,2), (3,3)\},
\{(4,1), (v-1,2), (3,3)\}, \\
\{(4,1), (v-1,2), (4,3)\},
\{(4,1), (0,2), (4,3)\},
\{(3,1), (0,2), (4,3)\},$

$\vdots$

$\{(v-2,1), (v-2,2), (v-3,3)\},
\{(v-2,1), (v-1,2), (v-3,3)\},\\
\{(v-2,1), (v-1,2), (v-2,3)\},
\{(v-2,1), (0,2), (v-2,3)\}, \\
\{(v-3,1), (0,2), (v-2,3)\},
\ldots,
\{(v-1,1), (v-2,2), (v-2,3)\}, \\
\{(v-1,1), (v-1,2), (v-2,3)\},
\{(v-1,1), (v-1,2), (v-1,3)\} $

\item[] $P_{3} = \{(v-2,1), (0,2), (v-1,3)\},
\{(v-2,1), (1,2), (v-1,3)\}, \\
\{(v-3,1), (1,2), (v-1,3)\},
\{(v-3,1), (2,2), (v-1,3)\},
\ldots, \\
\{(1,1), (v-3,2), (v-1,3)\},
\{(1,1), (v-2,2), (v-1,3)\}, \\
\{(0,1), (v-2,2), (v-1,3)\}$

\item[] $P_{4} = \{(v-2,1), (1,2), (0,3)\},
\{(v-2,1), (2,2), (0,3)\}, \\
\{(v-3,1), (2,2), (0,3)\},
\{(v-3,1), (3,2), (0,3)\},
\ldots, \\
\{(2,1), (v-3,2), (0,3)\},
\{(2,1), (v-2,2), (0,3)\},
\{(1,1), (v-2,2), (0,3)\}$

\item[] $P_{5} = \{(0,1), (v-1,2), (0,3)\},
\{(0,1), (0,2), (0,3)\},
\{(v-1,1), (0,2), (0,3)\}$

\item[] $P_{6} = \{(1,1), (v-1,2), (1,3)\}$

\end{itemize}

Our strategy is to show that by using the nine Type 4 triples, these carefully chosen six paths can be stitched together with the three Hamilton paths that are obtained from Type 1 triples in order to construct a Hamilton cycle in the $2$-BIG of $\mathcal{A}'$.

The following classical result by Smith is crucial for our construction. (See \cite{smith, thomason, tutte} for proofs.)

\begin{theorem} \label{smith}
Let $G$ be a $3$-regular graph and let $e \in E(G)$. Then there are an even number of Hamilton cycles in $G$ which pass through $e$.
\end{theorem}

We are now ready to prove the following result.

\begin{theorem} \label{3v}
If $v \geq 4$ and there exists a \textnormal{TTS($v$)} with a Hamiltonian \textnormal{$2$-BIG}, then there exists a \textnormal{TTS($3v$)} with a Hamiltonian \textnormal{$2$-BIG}.
\end{theorem}

\begin{proof}
Suppose $\mathcal{A}=(V, \mathcal{B})$ is a \textnormal{TTS($v$)} whose \textnormal{$2$-BIG} is Hamiltonian. 
By Theorem \ref{smith}, there are at least two Hamilton cycles $H_{1}$ and $H_{2}$ in the $2$-BIG of $\mathcal{A}$. Therefore there is a triple $\{a,b,c\}$ in $\mathcal{B}$ with the property that there are two Hamilton cycles in the 2-BIG of $\mathcal{A}$ that use a different pair of edges incident with $\{a,b,c\}$. Without loss of generality, suppose that $H_{1}$ uses the two edges that connect $\{a,b,c\}$ to the distinct triples containing the pairs $\{a,b\}$ and $\{a,c\}$ respectively, whereas $H_{2}$ uses the two edges that connect $\{a,b,c\}$ to the distinct triples containing the pairs $\{a,b\}$ and $\{b,c\}$ respectively. Relabel the points of $\mathcal{A}$ so that $a$ is mapped to $v-1$, $b$ is mapped to $0$, and $c$ is mapped to $1$.


Apply the $3v$ Construction to form $\mathcal{A}'=(V', \mathcal{B}')$. It is easy to see that each block in $\mathcal{B}'$ is a distinct triple, and that each pair of points of $V'$ appears in exactly two triples in $\mathcal{B}'$. So $\mathcal{A}'=(V', \mathcal{B}')$ is a simple twofold triple system.

The Type 1 triples can be partitioned into three sets, each corresponding to the set $\mathcal{B}$ excluding the triple $\{0, 1, v-1\}$. The $2$-BIG obtained from each of these sets is isomorphic to the $2$-BIG of $\mathcal{A}$ with one vertex deleted. Moreover, these three graphs each have at least two Hamilton paths, namely those which correspond to the remnants of $H_{1}$ and $H_{2}$. In the first of these three graphs, let $P_{7}$ be the remnant of $H_{1}$ so that $P_{7}$ is a Hamilton path from the block containing the pair $\{(0,1),(v-1,1)\}$ (which is adjacent to the Type $4$ triple $\{(0,1), (v-1,1), (v-1,3)\}$) to the block containing the pair $\{(1,1),(v-1,1)\}$ (which is adjacent to the Type $4$ triple $\{(1,1), (v-1,1), (0,3)\}$). In the second of these three graphs, let $P_{8}$ be the remnant of $H_{1}$ so that $P_{8}$ is a Hamilton path from the block containing the pair $\{(0,2),(v-1,2)\}$ (which is adjacent to the Type $4$ triple $\{(0,2), (v-1,2), (1,1)\}$) to the block containing the pair $\{(1,2),(v-1,2)\}$ (which is adjacent to the Type $4$ triple $\{(1,2), (v-1,2), (0,1)\}$). And in the third of these three graphs, let $P_{9}$ be the remnant of $H_{2}$ so that $P_{9}$ is a Hamilton path from the block containing the pair $\{(0,3),(v-1,3)\}$ (which is adjacent to the Type $4$ triple $\{(0,3), (v-1,3), (v-1,2)\}$) to the block containing the pair $\{(0,3),(1,3)\}$ (which is adjacent to the Type $4$ triple $\{(0,3), (1,3), (1,2)\}$).


Then a Hamilton cycle in the $2$-BIG of $\mathcal{A}'$ is given as $\big(P_{1}, \{(0,1), (1,1), \linebreak (1,3)\}, P_{6}, \{(1,1), (0,2), (v-1,2)\}, P_{8}, \{(0,1),(1,2), (v-1,2)\}, P_{5}, \{(v-1,1),\linebreak (0,2), (1,2)\},
P_{2}, \{(v-1,2),(0,3),(v-1,3)\}, P_{9}, \{(1,2),(0,3),(1,3)\}, P_{4}, \{(1,1),\linebreak (v-1,1),(0,3)\},
P_{7}, \{(0,1),(v-1,1),(v-1,3)\}, P_{3}, \{(0,2),(1,3),(v-1,3)\} \big)$. \end{proof}

\section{\boldmath $3v+1$ Construction when $v$ is odd}
\label{Sec:$3v+1$ Construction}
The $3v$ Construction in Section \ref{Sec:$3v$ Construction} can be extended to a $3v+1$ Construction when $v$ is odd. The critical observation here is that the quasigroup $Q_{2}$ of the $3v$ Construction has a transversal when $v$ is odd (in fact, each diagonal of $Q_{2}$ yields a transversal). The particular latin square that is being used guarantees the existence of a transversal only when $v$ is odd, although we note that the general methodology can be used for even values of $v$ if suitable latin squares can be found.

\textbf{$\mathbf{3\textit{v}+1}$ Construction}: Let $v\geq7$ be odd, and let $\mathcal{A}=(V, \mathcal{B})$ be a TTS($v$) on the point set $V=\mathbb{Z}_{v}$ where $ \{0, 1, v-1 \} \in \mathcal{B}$ (if necessary, the points of $V$ can be permuted so that such a block exists). Suppose that the $2$-BIG of $\mathcal{A}$ is Hamiltonian. Let $Q_{1}=(V,\circ_{1})$ and $Q_{2}=(V,\circ_{2})$ be two quasigroups of order $v$ where for $i,j \in \mathbb{Z}_{v}$, $i\circ_{1}j=i+j$ (modulo $v$) and $i\circ_{2}j=i+j+1$ (modulo $v$). Let $L_{1}$ and $L_{2}$ be the corresponding latin squares. Let $T$ be a transversal of $L_{2}$ (note that each diagonal of $L_{2}$ is a transversal).

Use the $3v$ Construction to form a TTS($3v$) $\mathcal{A}'=(V', \mathcal{B}')$ with a Hamiltonian $2$-BIG where $V'=V\times \{1,2,3\}$. Then to form a TTS($3v+1$) $\mathcal{A}''=(V'', \mathcal{B}'')$ where $V''=V'\cup \{\infty \}$, for each $(i,j)$ that is a cell of $T$, replace the triple $\{(i,1),(j,2),(i\circ_{2}j,3)\}$ with the triples $\{(i,1),(j,2),\infty \}$, $\{(i,1), \infty, (i\circ_{2}j,3)\}$, $\{\infty, (j,2),(i\circ_{2}j,3) \}$.

\bigskip

Now we proceed to proving that the TTS($3v+1$) that is obtained from the $3v+1$ Construction as described above has a Hamiltonian $2$-BIG.

\begin{theorem} \label{3v+1}
If $v \geq 7$ is odd and there exists a \textnormal{TTS($v$)} with a Hamiltonian \textnormal{$2$-BIG}, then there exists a \textnormal{TTS($3v+1$)} with a Hamiltonian \textnormal{$2$-BIG}.
\end{theorem}

\begin{proof} Let $\mathcal{A}=(V, \mathcal{B})$ be a TTS($v$) on the point set $V=\mathbb{Z}_{v}$. 
Apply the $3v$ Construction and use the proof of Theorem \ref{3v} to form a TTS$(3v)$ $\mathcal{A}'=(V', \mathcal{B}')$ whose $2$-BIG has a Hamilton cycle $H$. Next, follow the steps described in the $3v+1$ Construction to form a TTS($3v+1$) $\mathcal{A}''=(V'', \mathcal{B}'')$ from this TTS($3v$) $\mathcal{A}'$.

Using $H$ we find a Hamilton cycle $H^{*}$ in the $2$-BIG of $\mathcal{A}''$ as follows: For each $(i,j)$ that is a cell of the transversal $T$ in $Q_{2}$ 
such that in $H$ the triple preceding $\{(i,1),(j,2),(i\circ_{2}j,3)\}$ contains the pair $\{(i,1),(j,2)\}$, and the triple following $\{(i,1),(j,2),(i\circ_{2}j,3)\}$ contains the pair $\{(j,2),(i\circ_{2}j,3)\}$, replace $\{(i,1),(j,2),(i\circ_{2}j,3)\}$  with the triples $\{(i,1),(j,2),\infty \}$, $\{(i,1), \infty, (i\circ_{2}j,3)\}$, $\{\infty, (j,2),(i\circ_{2}j,3) \}$ in this order. Noting that these three triples induce a $K_{3}$ in the $2$-BIG of $\mathcal{A}''$, we observe that an appropriate ordering of these three triples can be similarly chosen in the two remaining cases according to the pairs that are included in the triples preceding and following $\{(i,1),(j,2),(i\circ_{2}j,3)\}$ in $H$. This procedure yields a Hamilton cycle $H^{*}$ in the $2$-BIG of $\mathcal{A}''$.
\end{proof}

\section{\boldmath $3v+3$ Construction when $v$ is odd}
\label{Sec:$3v+3$ Construction}
The $3v$ Construction in Section \ref{Sec:$3v$ Construction} can also be extended to a $3v+3$ Construction when $v$ is odd, however this requires much more delicate arguments. The idea in general is to start proceeding as in the $3v+1$ Construction, but this time three transversals in $Q_{2}$ are to be found. The set of pairs obtained from the triples coming from each transversal is then paired with one of the three new points to form new triples. The major difficulty here is that in order to complete the block set to the block set of a TTS($3v+3$) we need to include two copies of a triple where all three new points appear together, and doing so the $2$-BIG of this initial TTS($3v+3$) becomes disconnected (hence non-Hamiltonian). We overcome this issue by performing a second, very carefully chosen trade. We will show that the particular trade that is being performed guarantees that the $2$-BIG of the resulting TTS($3v+3$) is Hamiltonian.

As with the $3v+1$ Construction, the particular latin square that we are using guarantees the existence of transversals only when $v$ is odd, but the general methodology for the $3v+3$ Construction can potentially be used for even values of $v$ if suitable latin squares can be found that yield three transversals.

\textbf{$\mathbf{3\textit{v}+3}$ Construction}: Let $v\geq7$ be odd, and let $\mathcal{A}=(V, \mathcal{B})$ be a TTS($v$) on the point set $V=\mathbb{Z}_{v}$. 
Suppose that the $2$-BIG of $\mathcal{A}$ is Hamiltonian. Let $Q_{1}=(V,\circ_{1})$ and $Q_{2}=(V,\circ_{2})$ be two quasigroups of order $v$ where for $i,j \in \mathbb{Z}_{v}$, $i\circ_{1}j=i+j$ (modulo $v$) and $i\circ_{2}j=i+j+1$ (modulo $v$). Let $L_{1}$ and $L_{2}$ be the corresponding latin squares. Let $T_{1}$, $T_{2}$ and $T_{3}$ be three diagonal transversals of $L_{2}$ defined as $T_{1}= \{(a,b) \mid a=b\}$, $T_{2}= \{(a,b)\mid a=b+2\}$ and $T_{3}= \{(a,b) \mid a=b+1\}$ (where all arithmetic is done modulo $v$). Note that $T_{1}, T_{2}, T_{3}$ respectively contain the cells $(0,0)$, $(1,v-1)$, and $(1,0)$ (which will later give rise to the triples $\{\infty_{1}, (0,2), (1,3)\}$, $\{(1,1), \infty_{2}, (1,3)\}$ and $\{(1,1), (0,2), \infty_{3}\}$ among other triples).

Use the $3v$ Construction to form a TTS($3v$) $\mathcal{A}'=(V', \mathcal{B}')$ with a Hamiltonian $2$-BIG where $V'=V\times \{1,2,3\}$. Then to form a TTS($3v+3$) $\mathcal{A}''=(V'', \mathcal{B}'')$ where $V''=V'\cup \{\infty_{1}, \infty_{2}, \infty_{3} \}$, for each $(i,j)$ that is a cell of $T_{t}$ ($t= 1,2,3$), replace the triple $\{(i,1),(j,2),(i\circ_{2}j,3)\}$ with the triples $\{(i,1),(j,2),\infty_{t} \}$, $\{(i,1), \infty_{t}, (i\circ_{2}j,3)\}$, $\{\infty_{t}, (j,2),(i\circ_{2}j,3) \}$, except that we deliberately do not include the triples $\{\infty_{1}, (0,2), (1,3)\}$, $\{(1,1), \infty_{2}, (1,3)\}$, $\{(1,1), (0,2),\infty_{3}\}$. Moreover, include in $\mathcal{B}''$ the triples $\{(1,1), (0,2), (1,3)\}$, $\{\infty_{1}, \infty_{2}, (1,3)\}$, $\{\infty_{1}, \linebreak (0,2), \infty_{3}\}$, $\{(1,1), \infty_{2}, \infty_{3}\}$, $\{\infty_{1}, \infty_{2}, \infty_{3}\}$.

It is not difficult to see that each block in $\mathcal{B}''$ is a distinct triple (in particular, the triple $\{\{(1,1), (0,2), (1,3)\}$ was first traded away in the $3v$ Construction when $\mathcal{A}'=(V', \mathcal{B}')$ was being formed, and then it is included back in the block set of $\mathcal{A}''$), and that each pair of points of $V''$ appears in exactly two triples in $\mathcal{B}''$. So $\mathcal{A}''=(V'', \mathcal{B}'')$ is a simple twofold triple system.
\bigskip

Now we proceed to proving that the TTS($3v+3$) that is obtained from the $3v+3$ Construction as described above has a Hamiltonian $2$-BIG.

\begin{theorem} \label{3v+3}
If $v \geq 7$ is odd and there exists a \textnormal{TTS($v$)} with a Hamiltonian \textnormal{$2$-BIG}, then there exists a \textnormal{TTS($3v+3$)} with a Hamiltonian \textnormal{$2$-BIG}.
\end{theorem}

\begin{proof} Let $\mathcal{A}=(V, \mathcal{B})$ be a TTS($v$) on the point set $V=\mathbb{Z}_{v}$. 
Apply the $3v$ Construction and use the proof of Theorem \ref{3v} to form a TTS$(3v)$ $\mathcal{A}'=(V', \mathcal{B}')$ whose $2$-BIG has a Hamilton cycle $H$. Now form a TTS($3v+3$) $\mathcal{A}''=(V'', \mathcal{B}'')$ from $\mathcal{A}'$ by following the steps given in the $3v+3$ Construction.

Our aim is to use $H$ to find a Hamilton cycle $H^{*}$ in the $2$-BIG of $\mathcal{A}''$. This can be done as follows: Proceed similarly as in the $3v+1$ Construction. Let $(i,j)$ be a cell of the transversal $T_{t}$ ($t=1,2,3$) in $Q_{2}$ 
such that in $H$ the triple preceding $\{(i,1),(j,2),(i\circ_{2}j,3)\}$ contains the pair $\{(i,1),(j,2)\}$, and the triple following $\{(i,1),(j,2),(i\circ_{2}j,3)\}$ contains the pair $\{(j,2),(i\circ_{2}j,3)\}$. Replace $\{(i,1),(j,2),(i\circ_{2}j,3)\}$ with the triples $\{(i,1),(j,2),\infty_{t} \}$, $\{(i,1), \infty_{t}, (i\circ_{2}j,3)\}$, $\{\infty_{t}, (j,2),(i\circ_{2}j,3) \}$ in this order. (It was noted in the proof of Theorem $\ref{3v+1}$ that an appropriate ordering of these three triples can similarly be chosen in the two remaining cases according to the pairs that are included in the triples preceding and following $\{(i,1),(j,2),(i\circ_{2}j,3)\}$ in $H_{1}$.) Then delete the three triples $\{\infty_{1}, (0,2), (1,3)\}$, $\{(1,1), \infty_{2}, (1,3)\}$, $\{(1,1), (0,2),\infty_{3}\}$ that were excluded in the $3v+3$ Construction. In the $2$-BIG that is induced by the triples in
$(\mathcal{B}'' \setminus \{\{\infty_{1}, (0,2), (1,3)\}, \{(1,1), \infty_{2}, (1,3)\}, \{(1,1), (0,2),\infty_{3}\} \}) \linebreak \cup \{\{(1,1), (0,2), (1,3)\}, \{\infty_{1}, \infty_{2}, (1,3)\}, \{\infty_{1}, (0,2), \infty_{3}\}, \{(1,1), \infty_{2}, \infty_{3}\}, \{\infty_{1}, \linebreak \infty_{2}, \infty_{3}\}\}$, the deletion of these three triples has the effect of breaking a Hamilton cycle into three paths. These three paths can be identified as follows:

\begin{itemize}
\item[] $P'$ has ends $\{(1,1),\infty_{3},(2,3)\}$ and $\{(0,2),(1,3),(v-1,3)\}$,

\item[] $P''$ has ends $\{\infty_{3},(0,2),(2,3)\}$ and $\{\infty_{2},(v-1,2),(1,3)\}$, and

\item[] $P'''$ has ends $\{(0,1),(0,2),\infty_{1}\}$ and $\{(0,1),(1,1),(1,3)\}$.
\end{itemize}

\noindent Then a Hamilton cycle $H^{*}$ in the $2$-BIG of $\mathcal{A}''$ can be given as $H^{*}=(\{\infty_{1}, \infty_{2}, \linebreak \infty_{3}\},
\{\infty_{1}, \infty_{2}, (1,3)\}, P'', \{\infty_{1}, (0,2), \infty_{3}\}, P''', \{(1,1), (0,2), (1,3)\}, P', \{(1,1), \linebreak \infty_{2}, \infty_{3}\})$.
\end{proof}

\section{Main Results} \label{Sec:MainResults}
In this section we prove our main theorem, 
which completely settles the spectrum problem of TTSs with Hamiltonian $2$-BIG.

First we need to settle some small orders, which we will later use to start the induction that will prove Theorem \ref{mainresult}. 
We include an explicit example of a TTS($v$) with a Hamiltonian $2$-BIG for each order $v \in \{4,7,9,10,18\}$. In each example, the ordering of the triples identifies the Hamilton cycle.
\medskip

\textbf{$v=4$:} The unique TTS($4$) has $K_{4}$ as its $2$-BIG.

\begin{center}
$\begin{array}{cccc}
 (\{0, 1, 2\},  &  \{0, 1, 3\},  &  \{0, 2, 3\},  &  \{1, 2, 3\})
\end{array}$
\end{center}

\textbf{$v=7$:} The existence of a TTS($7$) with a Hamiltonian $2$-BIG is known by Theorem \ref{dewar}. We give an explicit example.

\begin{center}
$\begin{array}{ccccccc}
(\{0, 1, 2\},  & \{0, 1, 3\},  & \{0, 3, 5\},  & \{0, 5, 6\},  & \{0, 4, 6\},  & \{0, 2, 4\},  &  \{2, 4, 5\} \\
 \{2, 3, 5\},  & \{2, 3, 6\},  & \{3, 4, 6\},  & \{1, 3, 4\},  & \{1, 4, 5\},  & \{1, 5, 6\},  &  \{1, 2, 6\} )
\end{array} $
\end{center}

\textbf{$v=9$:} The particular TTS($9$) whose blocks are shown below has a Hamiltonian $2$-BIG. Moreover the $2$-BIG of this TTS($9$) is bipartite, and therefore this TTS($9$) is the union of two STSs of order $9$.

\begin{center}
$\begin{array}{cccccc}
 (\{0, 1, 2\},  &  \{0, 1, 3\},  &  \{0, 3, 4\},  &  \{0, 2, 4\},  &
  \{2, 4, 6\},  &  \{1, 4, 6\}, \\

  \{1, 4, 8\},  &  \{1, 7, 8\},   &
  \{1, 6, 7\},  &  \{2, 6, 7\},  &  \{2, 3, 7\},  &  \{3, 4, 7\}, \\

  \{4, 5, 7\},  &  \{4, 5, 8\},  &  \{2, 5, 8\},  &  \{2, 3, 8\},  &
  \{3, 6, 8\},  &  \{0, 6, 8\}, \\

  \{0, 7, 8\},  &  \{0, 5, 7\},  &
  \{0, 5, 6\},  &  \{3, 5, 6\},  &  \{1, 3, 5\},  &  \{1, 2, 5\}) \\
\end{array}$
\end{center}

\textbf{$v=10$:}  There are $394$ non-isomorphic simple TTSs of order $10$.  The largest girth for their $2$-BIGs is girth $7$, which occurs only for a single TTS($10$).  This particular TTS($10$) has a Hamilton cycle as follows.

\begin{center}
$\begin{array}{cccccc}
 (\{ 0, 1 ,2  \},  &  \{ 0, 1, 3 \},  &  \{ 0, 3, 5  \},  &  \{ 0, 5, 7  \},  &  \{ 0, 7, 9  \},  &  \{0, 8, 9   \}, \\

  \{ 0, 6, 8  \},  &  \{ 0, 4, 6  \},   &  \{ 4, 6, 9  \},  &  \{ 4, 5, 9 \},  &  \{ 1, 4, 5 \},  &  \{ 1, 3, 4  \}, \\

  \{ 3, 4, 7  \},  &  \{ 3, 7, 9  \},  &  \{ 3, 6, 9  \},  &  \{ 2, 3, 6  \},  &  \{ 2, 3, 8  \},  &  \{ 3, 5, 8  \}, \\

  \{ 5, 6, 8  \},  &  \{ 1, 5, 6  \},  &  \{  1, 6, 7 \},  &  \{  2, 6, 7 \},  &  \{ 2, 5, 7  \},  &  \{  2, 5, 9 \}, \\

  \{ 1, 2, 9  \},  &  \{ 1, 8, 9  \},  &  \{ 1, 7, 8  \},  &  \{ 4, 7, 8  \},  &  \{ 2, 4, 8  \},  &  \{ 0, 2, 4 \}) \\
\end{array}$
\end{center}
In the $2$-BIG of this TTS($10$)
\begin{center}
$\begin{array}{ccccccc}
 (\{ 0,1,2 \},  &  \{ 0,1,3\},  &  \{ 0,3,5\},  &  \{ 0,5,7 \},  &  \{ 2,5,7 \},  &  \{ 2,5,9 \} & \{ 1,2,9 \}) \\
\end{array}$
\end{center}
is a cycle of length $7$.

\medskip

\textbf{$v=18$:} The $2$-BIG of the following TTS($18$) is Hamiltonian and has girth~$5$.

\begin{center}
$\begin{array}{cccccc}
(\{ 0, 1, 2  \}, & \{  0, 1, 3 \}, & \{  0, 3, 5  \}, & \{ 0, 5, 6   \}, & \{ 0, 4, 6  \}, & \{ 0, 2, 4  \}, \\

 \{  2, 3, 4  \}, & \{ 1, 3, 4  \}, & \{  1, 4, 8  \}, & \{  1, 6, 8  \}, & \{  1, 5, 6  \}, & \{  1, 5, 7  \},  \\

 \{  4, 5, 7  \}, & \{ 4, 5, 13  \}, & \{  4, 13, 16  \}, & \{  4, 10, 16  \}, & \{ 2, 10, 16  \}, & \{  2, 10, 13  \}, \\

 \{  9, 10, 13  \}, & \{ 1, 9, 13  \}, & \{  1, 9, 12  \}, & \{ 1, 10, 12  \}, & \{  1, 10, 14  \}, & \{ 1, 14, 17  \},  \\

 \{  1, 16, 17  \}, & \{ 0, 16, 17  \}, & \{  0, 14, 16  \}, & \{  2, 14, 16  \}, & \{  2, 14, 15  \}, & \{ 3, 14, 15\},  \\

 \{ 3, 14, 17  \}, & \{  3, 13, 17  \}, & \{  5, 13, 17  \}, & \{5, 10, 17 \}, & \{5, 10, 14 \}, & \{  5, 9, 14  \},  \\

 \{ 4, 9, 14 \}, & \{  4, 6, 14  \}, & \{  6, 13, 14  \}, & \{ 8, 13, 14  \}, & \{ 8, 11, 14\}, & \{ 7, 11, 14  \},  \\

 \{ 7, 12, 14  \}, & \{ 0, 12, 14  \}, & \{ 0, 12, 13  \}, & \{8, 12, 13  \}, & \{8, 12, 15  \}, & \{7, 12, 15  \},  \\

 \{  4, 7, 15  \}, & \{4, 10, 15 \}, & \{  8, 10, 15\}, & \{ 0, 8, 10\}, & \{  0, 10, 11  \}, & \{ 6, 10, 11  \},  \\

 \{  2, 6, 11  \}, & \{  2, 6, 8  \}, & \{  2, 5, 8 \}, & \{ 3, 5, 8  \}, & \{ 3, 7, 8  \}, & \{ 0, 7, 8  \},  \\

 \{  0, 7, 9  \}, & \{  0, 9, 11  \}, & \{  4, 9, 11  \}, & \{ 4, 11, 12  \}, & \{  4, 12, 17 \}, & \{  4, 8, 17  \},  \\

 \{  8, 9, 17  \}, & \{  6, 9, 17  \}, & \{  6, 12, 17  \}, & \{  6, 12, 16  \}, & \{  5, 12, 16  \}, & \{  2, 5, 12  \},  \\

 \{ 2, 9, 12  \}, & \{  2, 3, 9  \}, & \{3, 9, 10 \}, & \{ 3, 10, 12 \}, & \{ 3, 11, 12 \}, & \{ 3, 11, 13 \},  \\

 \{ 1, 11, 13  \}, & \{ 1, 11, 15 \}, & \{ 1, 15, 16 \}, & \{ 3, 15, 16  \}, & \{  3, 6, 16  \}, & \{ 3, 6, 7  \},  \\

 \{ 6, 7, 10  \}, & \{  7, 10, 17  \}, & \{  7, 11, 17  \}, & \{ 2, 11, 17  \}, & \{  2, 15, 17  \}, & \{ 0, 15, 17  \},  \\

 \{ 0, 13, 15  \}, & \{  6, 13, 15  \}, & \{ 6, 9, 15  \}, & \{  5, 9, 15 \}, & \{  5, 11, 15  \}, & \{ 5, 11, 16  \},  \\

 \{ 8, 11, 16  \}, & \{  8, 9, 16 \}, & \{  7, 9, 16 \}, & \{  7, 13, 16  \}, & \{ 2, 7, 13  \}, & \{ 1, 2, 7 \})  \\
\end{array}$

\end{center}
In the $2$-BIG of this TTS($18$)
\begin{center}
$\begin{array}{ccccc}
 (\{ 0,1,2 \},  &  \{ 0,1,3\},  &  \{ 1,3,4\},  &  \{ 2,3,4 \},  &  \{0,2,4 \}) \\
\end{array}$
\end{center}
is a cycle of length $5$.
\begin{theorem} \label{mainresult}
Suppose that $v\geq3$. There exists a TTS($v$) with a Hamiltonian $2$-BIG if and only if $v \equiv 0,1$ (modulo $3$), except for $v \in \{3,6\}$.
\end{theorem}

\begin{proof}
First note that by \cite{bhattacharya} a TTS($v$) exists only if $v \equiv 0,1$ (modulo $3$). The $2$-BIG of the unique TTS($3$) consists of two isolated vertices, and therefore it is non-Hamiltonian. There is a unique simple TTS($6$), and it has a connected $2$-BIG, however it is not Hamiltonian. In fact, the $2$-BIG of the unique simple TTS($6$) is the Petersen graph \cite{mahmoodian}.

To prove the sufficiency of the condition, suppose that $v \equiv 0,1$ (modulo $3$) and note that an example for a TTS($v$) with a Hamiltonian $2$-BIG is already given for $v \in \{4,7,9,10,18\}$. Moreover a TTS($v$) with a Hamiltonian $2$-BIG exists for $v \in \{13,15,16,19\}$ by Theorem \ref{dewar}. For $v =12$ and $v=21$, a TTS($v$) with a Hamiltonian $2$-BIG can be constructed using the $3v$ Construction together with respectively a TTS($4$) and TTS($7$) whose $2$-BIGs are Hamiltonian. This establishes that for $7 \leq v \leq 21$, $v$ is in the spectrum of TTSs with a Hamiltonian $2$-BIG.

Let $v = 12k +r\geq 22$ (for some positive $k$, and non-negative $r$ such that $r\leq 10$), and inductively suppose that for all $t \equiv 0,1$ (modulo $3$) where $7 \leq t < v$ there exists a TTS($t$) with a Hamiltonian $2$-BIG. Then to construct a TTS($v$) with a Hamiltonian $2$-BIG one can apply the following procedure:

\begin{itemize}

\item[] If $r=0$, use the $3v$ Construction, unless $k \equiv 2$ (modulo $3$), in which case use the $3v+3$ Construction.

\item[] If $r=1$, Theorem \ref{dewar} guarantees the existence of a TTS($v$) with a Hamiltonian $2$-BIG, unless $k \equiv 2$ (modulo $5$), in which case use the $2v+1$ Construction.

\item[] If $r=3$ or $r=7$, Theorem \ref{dewar} guarantees the existence of a TTS($v$) with a Hamiltonian $2$-BIG.

\item[] If $r=4$, Theorem \ref{dewar} guarantees the existence of a TTS($v$) with a Hamiltonian $2$-BIG, unless $k \equiv 3$ (modulo $5$), in which case use the $2v+2$ Construction.

\item[] If $r=6$, use the $3v$ Construction, unless $k \equiv 0$ (modulo $3$), in which case use the $3v+3$ Construction.


\item[] If $r=9$, use the $3v$ Construction, unless $k \equiv 2$ (modulo $3$), in which case use the $2v+1$ Construction.

\item[] If $r= 10$ use the $3v+1$ Construction, unless $k \equiv 2$ (modulo $3$), in which case use the $2v+2$ Construction.
\end{itemize}
This finishes the proof. \end{proof}

We conclude with an easy consequence of Theorem \ref{mainresult}.

\begin{theorem}
Suppose that $v\geq3$. There exists a TTS($v$) with a $2$-BIG that has a Hamilton path if and only if $v \equiv 0,1$ (modulo $3$), except for $v=3$.
\end{theorem}

\begin{proof}
This follows from Theorem \ref{mainresult} and the observation that there is no simple TTS($3$), whereas the $2$-BIG of the unique simple TTS($6$) is the Petersen graph, which has a Hamilton path.
\end{proof}

\section{Acknowledgements}

A.~Erzurumluo\u{g}lu acknowledges research support from AARMS.
D.A.\ Pike acknowledges research grant support from NSERC (grant numbers RGPIN-04456-2016 and 217627-2010), CFI and IRIF, as well as computational support from Compute Canada and its consortia (especially ACENET, SHARCNET and WestGrid) and The Centre for Health Informatics and Analytics of the Faculty of Medicine at Memorial University of Newfoundland.

\end{document}